\documentclass[reqno]{amsart}
\usepackage{epsfig}
\usepackage{color}
\usepackage{amssymb,amsmath,amsthm,amstext,amsfonts}
\usepackage[mathscr]{eucal}
\usepackage{dsfont}

\usepackage{psfrag}
\usepackage{url}
\usepackage{epstopdf}
\usepackage{epic,eepic}
\usepackage{upgreek}
\usepackage{amssymb}
\usepackage{mathrsfs}
\usepackage{verbatim}
\usepackage{mathrsfs}
\usepackage{amstext}
\usepackage{amsthm}
\usepackage{amssymb}
\usepackage{graphicx}

\usepackage[colorlinks=true, linkcolor=blue, urlcolor=red, citecolor=blue, hyperindex]{hyperref}

\makeatletter \@addtoreset{equation}{section} \makeatother

\renewcommand\thetable{\thesection.\@arabic\c@table}

\theoremstyle{plain}
\newtheorem{maintheorem}{Theorem}

\newtheorem{maincorollary}{Corollary}

\newtheorem{theorem}{Theorem}[section]
\newtheorem{proposition}{Proposition}[section]
\newtheorem{lemma}{Lemma}[section]
\newtheorem{corollary}{Corollary}[section]
\newtheorem{definition}{Definition}[section]
\newtheorem{remark}{Remark}[section]

\newcommand{\de} {\delta}       \newcommand{\De}{\Delta}
\newcommand{\vep}{\varepsilon}

\newcommand{\si} {\sigma}

\newcommand{\supp}{\operatorname{supp}}

\newcommand{\graph}{\operatorname{graph}}

\newcommand{\Leb}{Leb}

\newcommand{\cP}{\mathcal{P}}

\newcommand{\cE}{\mathcal{E}}

\newcounter{main}

\title
{Volume lemmas for partially hyperbolic \\ endomorphisms and applications}

\author{A. Cruz and G. Ferreira and P. Varandas}
\address{Anderson Cruz, Centro de Ci\^encias Exatas e Tecnol\'ogicas, Universidade Federal do Rec\^oncavo da Bahia, Av. Rui Barbosa, s/n, 44380-000 Cruz das Almas, BA, Brazil}
\email{anderson.cruz@ufrb.edu.br}

\address{Giovane Ferreira, Departamento de Matem\'atica, Universidade Federal do Maranh\~ao\\
Av. dos Portugueses, 1966, Vila Bacanga, 65065-545 S�o Lu�s, MA}
\email{giovane.ferreira@ufma.br}

\address{Paulo Varandas, Departamento de Matem\'atica, Universidade Federal da Bahia\\
Av. Ademar de Barros s/n, 40170-110 Salvador, Brazil}
\email{paulo.varandas@ufba.br}

\begin{document}

\begin{abstract}
We consider partially hyperbolic attractors for non-singular endomorphisms admitting an invariant stable bundle 
and a positively invariant cone field with non-uniform cone expansion at a positive Lebesgue measure
set of points. 
We prove volume lemmas for both Lebesgue measure on the topological basin of 
the attractor and the SRB measure supported on the attractor.
As a consequence under a mild assumption 
we prove exponential 
large deviation bounds for the convergence of Birkhoff averages associated to continuous observables with respect to 
the SRB measure. 
\end{abstract}


\maketitle

\date{\today}

\section{Introduction}

One of the main goals of ergodic theory is to describe the  
statistical properties of dynamical systems using invariant measures. 
In particular, the thermodynamic formalism aims the construction and description of the 
statistical properties of invariant measures that are physically relevant, meaning 
equilibrium states with respect to some potential.
Among these, one should refer the SRB measures whose construction and 
the study of their statistical properties rapidly became a topic of interest of the mathematics and physics communities. 
Unfortunately, apart from the uniformly hyperbolic setting, where the existence of finite Markov partitions allows to semiconjugate
the dynamics to subshifts of finite type (see e.g. \cite{Bow75}) there is no systematic approach
for the construction of SRB measures.

In respect to that, there have been important contributions to the study of (non-singular) endomorphisms. 
In the mid seventies, Przytycki~\cite{Prz76} extended the notion of uniform hyperbolicity to the context of endomorphisms and studied Anosov endomorphisms. Here, due to the non-invertibility of the dynamics, the existence of an invariant unstable subbundle for
uniformly hyperbolic basic pieces needs to be replaced by the existence of positively invariant cone fields on which vectors are uniformly expanded by all positive iterates (we refer the reader to Subsection~\ref{sec:UH} for more details). 
In \cite{QZ02}, Qian and Zhu constructed SRB measures
for Axiom A attractors of endomorphisms,
obtaining these as equilibrium states for the geometric potential 
defined using inverse limits. A characterization of SRB measures for uniformly hyperbolic
endomorphisms can also be given in terms of dimensional characteristics of the stable
manifold (cf. \cite{UW04}). 
The thermodynamic formalism of hyperbolic basic pieces for endomorphisms
had the contribution of Mihailescu and Urbanski~\cite{MU} 
that introduced and constructed inverse SRB measures for hyperbolic attractors of endomorphisms.
Among the difficulties that arise when dealing with non-invertible hyperbolic dynamics one should refer 
that unstable manifolds may have a complicated geometrical structure 
(cf. \cite{QXZ09,UW04}). 

The ergodic theory for endomorphisms beyond the scope of uniform hyperbolicity is much incomplete. 
In the context of surface endomorphisms, a major contribution is due to Tsujii \cite{Tsu05}, that proved
that for $r\ge 19$, $C^r$-generic partially hyperbolic endomorphisms 
with an unstable cone field admit finitely many SRB measures whose ergodic basins of attraction cover Lebesgue 
almost every point in the manifold (the regularity can be dropped to $r\ge 2$ on the space of non-singular 
endomorphisms). While it remains unknown if these SRB measures are (generically) hyperbolic, an extension to higher 
dimension should present significant difficulties as discussed by the author (cf.~\cite[page 43]{Tsu05}).
In \cite{CV17} the first and the third authors constructed SRB measures for robust classes classes
of non-singular $C^2$-partially hyperbolic endomorphisms, proved their statistical stability and the continuous dependence 
of their entropies with respect to the dynamics.  

Our purpose here is to contribute to the ergodic theory of partially hyperbolic attractors for endomorphisms. 
Given the non-invertible structure of the attractor, even in the case of hyperbolic endomorphisms 
unstable manifolds have self intersections (hence are badly understood in the ambient space) 
and stable manifolds may fail to have good geometric properties 
when lifted to the natural extension 
(see e.g \cite{Mih12b} for a detailed survey on geometric and ergodic aspects of hyperbolic sets for endomorphisms).
In the non-hyperbolic context the situation becomes even more intricate.
Here we consider attractors for $C^{1+\alpha}$ ($\alpha>0$) non-singular endomorphisms on compact Riemannian manifolds that admit uniform
contraction along a well defined invariant stable subbundle and has a positively invariant cone field.
Under a mild non-uniform cone-expansion assumption, these endomorphisms admit finitely many hyperbolic and physical 
SRB measures, these are statistically stable and their metric entropies vary continuously with respect to the underlying
dynamics, and the SRB measure is unique if the attractor is transitive (here we mean SRB measures as measures which are absolutely continuous desintegrations 
with respect to Lebesgue measure along Pesin unstable manifolds) ~\cite{CV17}.
In this context, unstable subbundles depend on the pre-orbits and are almost everywhere defined via Oseledets theorem, 
hence the unstable Jacobian is defined just almost everywhere (with respect to the SRB measure) and could fail to be 
H\"older continuous on the attractor. 
Our main contribution here is to overcome such lack of regularity of the unstable Jacobian we recover bounded distortion
and obtain two volume lemmas at instants of hyperbolicity ($c$-cone-hyperbolic times defined in \cite{CV17}).
The (invariant) SRB measure is supported on the attractor, and the volume lemma for the SRB measure lies on the
regularity of the unstable Jacobian along preimages of unstable disks. 
The volume lemma for Lebesgue measure on the topological basin of the attactor is even more subtle. 
Since local unstable manifolds may intersect we make use of some 'fake foliations' tangent to the cone field and,  
in order to use a Fubini-like argument, one needs to estimate the largest possible distortion using \emph{any} subspaces in 
the cone field  (cf. Proposition~\ref{lem. distorcao volume 2} for the 
precise statement) and its dependence on a small
open neighborhood. 
These volume lemmas are the main results of the paper. As a consequence we derive large deviations upper bounds for the
velocity of convergence of Birkhoff averages associated to continuous observables using sub-additive observables
associated to the invariant cone field. 

This paper is organized as follows. In Section~\ref{sec:Main} we describe the class of partially hyperbolic attractors considered here and state our results on large deviations. We recall some preliminary results in Section \ref{sec:Preliminaries}. In Section~\ref{sec:volume1} we prove a first volume lemma, for the Lebesgue 
measure on the basin on the attractor. The volume lemma for the SRB measure is proven in Section~\ref{sec:volume2}.
In Section~\ref{sec:LDP} we prove the large deviation bounds for the volume and SRB measure and, finally we give some examples at Section~\ref{sec:examples}.

\section{Statement of the main result}\label{sec:Main}

\subsection{Setting}
Throughout, let $M$ be a compact connected Riemannian $d^*$-dimensional manifold.
Assume that $f:M\rightarrow M$ is a $C^{1+\alpha}$, $\alpha>0$, non-singular endomorphism  
and that  $\Lambda\subset M$ a compact positively $f$-invariant subset.
Let $U\supset\Lambda$ be an open set such that $f(\overline U) \subset U$ and assume
$$
\Lambda= \Lambda_f:=\bigcap_{n\geq0}f^{n}(\overline{U}),
$$
where $\overline{U}$ denotes the closure of the set $U$.
We say that $\Lambda$ is a \emph{partially hyperbolic attractor} if
 there are a continuous splitting of the tangent bundle
$T_{U}M=E^{s}\oplus F$ (where $F$ is not necessarily $Df$-invariant)
and constants $ c>0$ and $0<\lambda<1$ satisfying:
\begin{enumerate}
\item [(H1)]$Df(x)\cdot E_{x}^{s}=E_{f(x)}^{s}$
for every $x\in U$;\label{enu:invstabledirection}
\item [(H2)]$\|Df^{n}(x)|_{E^{s}_x}\|\leq\lambda^{n}$ for every $x\in U$ and $n\in\mathbb{N}$;\label{enu:contstabledirection}
\item [(H3)] there is a cone field $U\ni x\mapsto C(x)$ of constant dimension $\dim F$ so that $Df(x) (C(x))\subseteq C(f(x))$
for every $x\in U$, and there is a positive Lebesgue measure set $H\subset U$ 
so that: 
\begin{equation}\label{eq:nuexpansion}
\limsup_{n\to\infty}\frac{1}{n}\sum_{j=0}^{n-1}\log\|(Df(f^{j}(x))|_{C(f^{j}(x))})^{-1}\|\leq-2c<0,
\end{equation}
 for every $x\in H$\emph{\label{enu:conenue}} (the expression $\|(Df(f^{j}(x))|_{C(f^{j}(x))})^{-1}\|$ is made precise at Subsection~\ref{sec:Preliminaries});
\item [(H4)]$\left\Vert Df(x)\, v\right\Vert \,\left\Vert Df(x)^{-1} w\right\Vert \leq\lambda\,\left\Vert v\right\Vert \,\left\Vert w\right\Vert $ for every $v\in E_{x}^{s}$, all $w\in Df(x) ( C(x))$
and $x\in U$.\label{enu:domination}
\end{enumerate}
Throughout the paper the constant $c>0$ will be fixed as above.
 We say that $\Lambda$ is an \emph{attractor},
$E^{s}$ is a \emph{uniformly contracting bundle} and that $C$ is a
 \emph{non-uniformly expanding cone field}(or that $C$ exhibits non-uniform expansion). 
 We will denote 
 $d:=\dim(F)=\dim(C)$.  

\begin{theorem}\label{thm:TEO1} \cite[Theorem~A]{CV17}
Let $f$ be a $C^{1+\alpha}$ non-singular endomorphism, $\Lambda$ be a partially hyperbolic attractor and
$U$ be the basin of attraction of $\Lambda$. 
Then there are finitely 
many SRB measures for $f$ whose basins of attraction cover $H$,
Lebesgue mod zero. Moreover, these measures are hyperbolic and, if $Leb(U\setminus H)=0$ then there are finitely many 
SRB measures for $f$ in $U$. Furthermore,
if $f\mid_\Lambda$ is transitive then there is a unique SRB measure for $f$ in $U$.
\end{theorem}

Many large deviation estimates for non-uniformly hyperbolic dynamics rely on the differentiability of the
pressure function on regular (e.g. H\"older continuous) potentials, which is still an open question in our context. 
The level-1 large deviations estimates that we shall obtain here stand for \emph{continuous} observables and rely on the following volume lemmas, which are the main results in this paper and of independent interest.

\medskip
\noindent {\bf First volume lemma:} \emph{
There exists $\delta>0$ so that for every $0<\vep<\delta$ there exists $K(\vep)>0$ so that if $n$ is a $c$-cone-hyperbolic time for $x$ 
and $F_1$ is {\bf any} 
subspace of dimension $d$ contained in $C(x)$ then
\begin{equation*}
K(\vep)^{-1} e^{-\log |\det Df^n(x)\mid_{F_1}|} \le \Leb(B(x,n,\vep)) \le K(\vep) e^{-\log |\det Df^n(x)\mid_{F_1}|}.
\end{equation*}
}

\medskip
\noindent {\bf Second volume lemma:} \emph{
Assume that the density of the SRB measure $\mu$ is bounded away from zero and infinity.
There exists $\delta>0$ so that for every $0<\vep<\delta$ there exists $K(\vep)>0$ so that if $n$ is a $c$-cone-hyperbolic time for $x$
and $F_1$ is {\bf any}  subspace of dimension $d$ contained in $C(x)$ 
then
\begin{equation*}
K(\vep)^{-1} e^{-\log |\det Df^n(x)\mid_{F_1}|} \le \mu(\Lambda \cap B(x,n,\vep)) \le K(\vep) \label{key}e^{-\log |\det Df^n(x)\mid_{F_1}|}.
\end{equation*}
}

Some comments are in order. First, the precise statements of the volume lemmas appear as Propositions~\ref{le:volume1} and 
~\ref{le:volume2}, respectively. Moreover, since invariance of subbundles in the cone field are obtained when the entire pre-orbit  
is fixed, the observable $\log |\det Df^n(x)\mid_{F_1}|$ above leads to a non-stationary Birkhoff sum associated to the family of potentials
$\{\log |\det Df\mid_{Df^j(F_1)}|\}_{0\le j \le n-1}$. We overcome this fact by proving that the asymptotic behavior of this family
of potentials coincides with the sub-additive behavior of the maximal volume computed using subspaces in the cone field.  More precisely
the sequence of functions $\Gamma_n(x)=\max_{F_1 \subset C(x)} |\det Df^n(x)\!\mid_{F_1}|$
 is sub-multiplicative and, by Kingman's subadditive ergodic theorem, if $\nu$ is an invariant probability measure then
\begin{equation}\label{def:Gamma}
\Gamma_\nu(x) = \lim_{n\rightarrow \infty}  \frac1n \log  \, \Gamma_n(x) 
	\quad \text{exists for $\nu$-a.e. $x$},
\end{equation}
and 
$$
\int \Gamma_\nu(x) \, d\nu = \inf_{n\ge 1}  \frac1n \int \log  \, \Gamma_n(x) \, d\nu(x),
$$ 
and these are related with the volume expansion of any subspace of fixed dimension contained
in the cone field (cf. Corollary~\ref{cor:hyp-volume}). 
Moreover, the latter are almost everywhere
constant provided that $\nu$ is ergodic. The following result asserts that under a mild assumption on the sequence of 
cone-hyperbolic times (see Definition~\ref{def:non-lacunar} for the definition of non-lacunarity) the SRB measure is
a weak Gibbs measure (see Proposition~\ref{co:measure.control} for the precise statement) and satisfies 
large deviations estimates for continuous observables:

\begin{maintheorem}\label{thm:main}
Let $f$ be a $C^{1+\alpha}$ non-singular endomorphism,
$\Lambda$ be a transitive partially hyperbolic attractor, $U$ be the basin of attraction of $\Lambda$
and $\mu$ be the unique SRB measure for $f\mid_\Lambda$. 
Assume that the density of the SRB measure is bounded away from zero and infinity and that the 
sequence of cone-hyperbolic times is non-lacunar for $\mu$-almost everywhere. Then: (i) $\mu$ is a weak Gibbs measure, and
(ii) there exists $C>0$ so that for every continuous observable $\phi$ and every closed interval  $F\subset \mathbb R$ 
\begin{align}
\limsup_{n\to\infty} \frac1n \log &\, {\mu} \,\Big(x\in \Lambda \colon \frac1n \sum_{j=0}^{n-1} \phi (f^j(x)) \in F \Big) 
	  \le \max\{\inf_{\beta>0}{\mathcal E}_\mu(\beta), -\inf_{c\in F} I(c) +\beta \} \label{LDP-SRB},
\end{align}
where 
$
{\mathcal E}_\mu(\beta)=\limsup_{n\to\infty} \frac1n \log \mu\Big(x \colon n_1( x) >\frac{\beta n}{2\log C} \Big)
$
and
$$
I(c)=	
\sup\Big\{\int \Gamma_\nu(x)\, d\nu(x) - h_{\nu}(f)
	: \nu \in \mathcal{M}_f, \int\phi \,d\nu =c \Big\} \ge 0.
$$ 
\end{maintheorem}

In rough terms, the large deviations estimate ~\eqref{LDP-SRB} bounds the velocity of convergence of Birkhoff averages 
as the smallest when comparing the exponential decay rate ${\mathcal E}_\mu(\beta)$ for the tail of the first $c$-cone-hyperbolic time map  and a kind of ``distance" $I(\cdot)$ from a set of measures to the equilibrium, computed among measures with the prescribed averages in the closed interval $F$. 
The proof of Theorem~\ref{thm:main} explores the fact that under the previous assumptions the SRB measures is 
a weak Gibbs measure (cf. Proposition~\ref{co:measure.control}). For a some account on large deviations
for endomorphisms we refer the reader to \cite{AP06, CT, LQZ03} and references therein.

\begin{remark}
If one assumes that the sequence of cone-hyperbolic times is non-lacunar for $\Leb$-almost every point then we also obtain the upper bound
\begin{align}
\limsup_{n\to\infty} \frac1n \log & \,\Leb \, \Big(x\in U \colon \frac1n \sum_{j=0}^{n-1} \phi (f^j(x)) \in F 
\Big) 
	\le \max\{\inf_{\beta>0}{\mathcal E}_\mu(\beta), -\inf_{c\in F 
	} I(c) \} 	 \label{LDP-LEB}
\end{align}
for every continuous observable $\phi$ and every $c\in \mathbb R$. However, while in the case of the SRB measure 
the non-lacunarity of the sequences of cone-hyperbolic times follows from integrability of the first cone-hyperbolic time map 
this is far from immediate in the case of non-invariant probability measures (see Remark~\ref{rmk:nonlacLeb} for a more detailed
discussion).
\end{remark}

In view of Theorem~\ref{thm:main} we can now provide a criterion for exponential large deviations. Note that 
the first cone-hyperbolic time map $n_1$ is well defined $\mu$-almost everywhere. We say the first cone-hyperbolic time 
map $n_1$ has \emph{exponential tail} (with respect to $\mu$) if there exists $\eta>0$ so that $\mu(n_1>n) \le e^{-\eta n}$ for all 
$n\gg 1$.

\begin{maincorollary}\label{cor:A}
In the context of Theorem~\ref{thm:main},
if the first cone-hyperbolic time map $n_1$ has exponential tail with respect to $\mu$ 
then for every $\delta>0$ it follows that 
$$
\limsup_{n\to\infty} \frac1n \log \mu \Big(x\in \Lambda \colon \big| \frac1n \sum_{j=0}^{n-1} \phi(f^j(x)) -\int \phi\, d\mu\big| 
\ge\delta  \Big)  <0.
$$
\end{maincorollary}


\section{Preliminaries}\label{sec:Preliminaries}
The present section is devoted to some preliminary discussion on the notion of uniform and non-uniform
hyperbolicity for non-singular endomorphisms. 
Given the vector spaces  $V$ and $W$ and a cone $E\subset V$, let the vector space 
$
\mathcal{L}_{E}(V,W):=\left\{ T\mid_{E}:\ T\in\mathcal{L}(V,W)\right\} 
$
be endowed with the norm $\|T|_{E}\|=\sup_{0\neq v\in E}\frac{\|T\cdot v\|}{\|v\|}$.
Given a non-singular endomorphism $f: M \to M$, a $Df$-invariant and convex cone field $C$ and $x\in M$, denote by
$(Df(x)|_{C(x)})^{-1}$ the map 
$$
Df(x)^{-1}\mid_{Df(x) (C(x))} : Df(x) (C(x)) \to C(x) \subset T_x M.
$$
If $TM=E\oplus F$ is a continuous splitting of the tangent bundle, the \emph{cone field of width $a>0$ centered at $F$}
is the continuous map $x\mapsto C(x)$ that assigns to each $x\in M$ the cone of width $a>0$ centered at $F_{x}$, where
$C:=\left\{  v=v_{1}\oplus v_{2} \in E\oplus F \colon \|v_{1}\|\leq a\|v_{2}\|\right\} $.
Finally, a submanifold $D\subset M$ is \emph{tangent to
the cone field $x\mapsto C(x)$ }if $\dim(D)=\dim(C)$
and $T_{x}D\subset C(x)$, for every $x\in D$.

Let $\Lambda\subset M$ be a compact positively invariant subset of
$M$ and $\Lambda\ni x\mapsto C(x)$ be a $Df$-invariant cone field on $\Lambda$.

\begin{definition}
Let $M$ be a compact Riemannian manifold, $f:M\rightarrow M$
be a non-singular endomorphism and $c>0$. We say that $n\in\mathbb{N}$
is a \emph{$c$-cone-hyperbolic time} for $x\in M$ (with respect
to $C$) if $Df(f^{j}(x)) ( C(f^{j}(x)))\subset C(f^{j+1}(x))$
for every $0\leq j\leq n-1$ and 
\begin{equation}
\prod_{j=n-k}^{^{n-1}}\left\Vert (Df(f^{j}(x))|_{C(f^{j}(x))})^{-1}\right\Vert \leq e^{-ck}\label{eq:hyptimes}
	\quad \text{ for every $1\leq k\leq n$.}
\end{equation}
\label{def:hyptime}
\end{definition}

It is easy to see that if $n<m$ are $c$-cone hyperbolic times then $m-n$ is a $c$-cone hyperbolic 
time for $f^{n}(x)$. Moreover, if $n$ is a $c$-cone-hyperbolic time for $x\in M$ then 
\begin{eqnarray}\label{eq ineq times}
	\Vert Df^{n-j}(f^{j}(x))\cdot v\Vert\geq e^{c(n-j)}\|v\|,
\end{eqnarray}
for every $v\in C(f^{j}(x))$ and $0\le j \le n-1$ (cf. \cite{CV17}).

\subsection{Natural extension\label{subsec:naturalextension}}

The natural extension of $M$ by $f$
is the set
$$
M^{f}:=\{ \hat{x}=(x_{-j})_{j\in\mathbb{N}} : x_{-j}\in M\mbox{ and } f(x_{-j})=x_{-j+1},\mbox{ for every }j\in\mathbb{N}\} .
$$
We can induce a metric on $M^{f}$ by
$
\hat{d}(\hat{x},\hat{y}):=\sum_{j\in\mathbb{N}}2^{-j}d(x_{-j},y_{-j}).
$
The space $M^{f}$ endowed with the metric $\hat{d}$ is a compact metric space and the topology
induced by $\hat{d}$ is equivalent to the topology induced by $M^{\mathbb{N}}$ 
endowed with Tychonov's product topology. The projection
$\pi:M^{f}\rightarrow M$ given by $\pi(\hat{x}):=x_{0}$ is a
continuous map. The lift of $f$ is the map $\hat{f}:M^{f}\rightarrow M^{f}$
given by
$
\hat{f}(\hat{x}):=(\dots,x_{-n},\dots,x_{-1},x_{0},f(x_{0})),
$
and it is clear that $f\circ\pi=\pi\circ\hat{f}$. 
Finally, for each $\hat{x}\in M^{f}$ we take $T_{\hat{x}}M^{f}:=T_{\pi(\hat{x})}M$
and set 
\[
\begin{array}{rccc}
D\hat{f}(\hat{x}):  & T_{\hat{x}}M^{f} & \rightarrow & T_{\hat{f}(\hat{x})} M^{f} \\
 & v & \mapsto & Df(\pi(\hat{x}))v.
\end{array}
\]
If $\Lambda\subset M$ is a compact positively invariant, that is,
$f(\Lambda)\subset\Lambda$, the natural extension
of $\Lambda$ by $f$ is the set of pre-orbits that lie on $\Lambda$, that is, 
$
\Lambda^{f}:=\{ \hat{x}=(x_{-j})_{j\in\mathbb{N}}:\ x_{-j}\in\Lambda\mbox{ and }f(x_{-j})=x_{-j+1}\mbox{ for every }j\in\mathbb{N}\} .
$
The projection $\pi$ induces a continuous bijection between $\hat{f}$-invariant probability measures and $f$-invariant probability measures: assigns 
every $\hat{\mu}$ in $M^{f}$ to the push forward 
$\pi_{*}\hat{\mu}=\mu$ (see e.g. \cite{QXZ09} for more details).

\subsection{Non-uniform hyperbolicity for endomorphisms}\label{sec:UH}

\subsubsection{Lyapunov exponents and non-uniform hyperbolicity}

We need the following:
\begin{proposition}\cite[Proposition I.3.5]{QXZ09}\label{prop:existsrb}
\label{prop:oseledet_nat_ext} 
Suppose that $M$ is a compact Riemannian
manifold. Let $f:M\rightarrow M$ be a $C^{1}$ map preserving a probability measure $ \mu$. There is a Borelian
set $\hat{\Delta}\subset M^{f}$ with $\hat{f}(\hat{\Delta})=\hat{\Delta}$
and $\hat{\mu}(\hat{\Delta})=1$ satisfying that for every
$\hat{x}\in\hat{\Delta}$ there is a splitting 
$T_{\hat{x}}M^{f}=E_{1}(\hat{x})\oplus\dots\oplus E_{r(\hat{x})}(\hat{x})$
and numbers
$+\infty>\lambda_{1}(\hat{x})>\lambda_{2}(\hat{x})>\dots>\lambda_{r(\hat{x})}(\hat{x})>-\infty$ (Lyapunov exponents) and multiplicities $m_{i}(\hat{x})$ for $1\leq i\leq r(\hat{x})$
such that:
\begin{enumerate}
\item $D\hat{f}(\hat{f}^{n}(\hat{x})):T_{\hat{f}^{n}(\hat{x})} M^{f}\rightarrow T_{\hat{f}^{n+1}(\hat{x})} M^{f}$
is a linear isomorphism for every $n\in\mathbb{Z}$;
\item the functions $r:\hat{\Delta}\rightarrow\mathbb{N}$, $m_{i}:\hat{\Delta}\rightarrow\mathbb{N}$ and 
$\lambda_{i}:\hat{\Delta}\rightarrow\mathbb{R}$ are $\hat{f}$-invariant ;
\item $\dim(E_{i}(\hat{x}))=m_{i}(\hat{x})$
for every $1\leq i\leq r(\hat{x})\le \dim M $;
\item the splitting is $D\hat{f}$-invariant, that is, $D\hat{f}(\hat{x})\cdot E_{i}(\hat{x})=E_{i}(\hat{f}(\hat{x}))$,
for every $1\leq i\leq r(\hat{x})$;
\item $\lim_{n\to\pm\infty}\frac{1}{n}\log\|D\hat{f}^{n}(\hat{x})\cdot u\|=\lambda_{i}(\hat{x})$
for every $u\in E_{i}(\hat{x})\backslash\left\{ 0\right\} $
and for every $1\leq i\leq r(\hat{x})$;
\item if
$
\rho_{1}(\hat{x})\geq\rho_{2}(\hat{x})\geq\dots\geq\rho_{d} (\hat{x})
$
 represent the numbers $\lambda_{i}(\hat{x})$ repeated
$m_{i}(\hat{x})$ times for each $1\leq i\leq r(\hat{x})$
and $\left\{ u_{1},u_{2},\dots,u_{d}\right\} $is a basis for $T_{\hat{x}}M^{f}$
satisfying
$
\lim_{n\to\pm\infty}\frac{1}{n}\log\|D\hat{f}^{n}(\hat{x})\cdot u_{i}\|=\rho_{i}(\hat{x})
$
for every $1\leq i\leq d$, then for any subsets $P,Q\subset\left\{ 1,2,\dots,d\right\} $ with
$P\cap Q=\emptyset$ one has
$
\lim_{n\to\infty}\frac{1}{n}\log\angle(D\hat{f}^{n}(\hat{x})\cdot E_{P},D\hat{f}^{n}(\hat{x})\cdot E_{Q})=0
$
where $E_{P}$ and $E_{Q}$ are, respectively, the subspaces generated
by $\left\{ u_{i}\right\} _{i\in P}$ and $\left\{ u_{i}\right\} _{i\in Q}$.
\end{enumerate}
Moreover, if $\hat\mu$ is ergodic then the previous functions are constant $\hat\mu$-almost everywhere.
\end{proposition}

Given an $f$-invariant probability measure $\mu$, we say that $\mu$ is \emph{hyperbolic} if is has no zero Lyapunov exponents. 
Associated to this formulation of the concept of Lyapunov exponents
we have the existence of unstable manifolds(stable manifolds, see \cite[Proposition V.4.5]{QXZ09} ) for almost every point with respect to some invariant
measure with some positive Lyapunov exponents. Consider the subspaces 
$
E_{\hat{x}}^{u}:=\bigoplus_{\lambda_{i}(\hat{x})>0}E_{i}(\hat{x})\text{ and }E_{\hat{x}}^{cs}:=\bigoplus_{\lambda_{i}(\hat{x})\leq0}E_{i}(\hat{x}).
$

\begin{proposition}\cite[Proposition V.4.4]{QXZ09}
\label{prop:unstable_manifold} There is a countable number of compact
subsets $(\hat{\Delta}_{i})_{i\in\mathbb{N}}$, of $M^{f}$ with $\bigcup_{i\in\mathbb{N}}\hat{\Delta}_{i}\subset\hat{\Delta}$
and $\hat{\mu}(\hat{\Delta}\backslash\bigcup_{i\in\mathbb{N}}\hat{\Delta}_{i})=0$
such that:
\begin{enumerate}
\item for each $\hat{\Delta}_{i}$ there is $k_{i}\in\mathbb{N}$ satisfying
$\dim(E^{u}(\hat{x}))=k_{i}$ for every $\hat{x}\in\hat{\Delta}_{i}$ and the subspaces
$E_{\hat{x}}^{u}$ and $E_{\hat{x}}^{cs}$ depend continuously on
$\hat{x}\in\hat{\Delta}_{i}$;
\item for any $\hat{\Delta}_{i}$ there is a family of $C^{1}$ embedded
$k_{i}$-dimensional disks $\left\{ W_{loc}^{u}(\hat{x})\right\} _{\hat{x}\in\hat{\Delta}_{i}}$
in $M$ and real numbers $\lambda_{i}$, $\vep\ll\lambda_{i}$,
$r_{i}<1$, $\gamma_{i}$, $\alpha_{i}$ and $\beta_{i}$ such that
the following properties hold for each $\hat{x}\in\hat{\Delta}_{i}$:
\begin{enumerate}
\item there is a $C^{1}$ map $h_{\hat{x}}:O_{\hat{x}}\rightarrow E_{\hat{x}}^{cs}$,
where $O_{\hat{x}}$ is an open set of $E_{\hat{x}}^{u}$ which contains\linebreak
$\left\{ v\in E_{\hat{x}}^{u}:\ \|v\|<\alpha_{i}\right\} $ satisfying: (i)
$h_{\hat{x}}(0)=0$ and $Dh_{\hat{x}}(0)=0$; (ii) Lip$(h_{\hat{x}})\leq\beta_{i}$ and \linebreak $Lip(Dh_{\hat{x}}(\cdot))\leq\beta_{i}$;
and (iii) $W_{loc}^{u}(\hat{x})=\exp_{x_{0}}(graph(h_{\hat{x}}))$.
\item for any $y_{0}\in W_{loc}^{u}(\hat{x})$ there is a unique
$\hat{y}\in M^{f}$ such that $\pi(\hat{y})=y_{0}$ and
$$
dist(x_{-n},y_{-n})\leq \min\{ r_{i}e^{-\vep_{i}n}, \gamma_{i}e^{-\lambda_{i}n}dist(x_{0},y_{0})\}, \; \text{ for every }n\in\mathbb{N},
$$
\item if
$
\hat{W}_{loc}^{u}(\hat{x}):=\big\{ \hat{y}\in M^{f}:\ y_{-n}\in W_{loc}^{u}(\hat{f}^{-n}(\hat{x})),\text{ for every }n\in\mathbb{N}
\big\}$
then $\pi:\hat{W}_{loc}^{u}(\hat{x})\rightarrow W_{loc}^{u}(\hat{x})$
is bijective and $\hat{f}^{-n}(\hat{W}_{loc}^{u}(\hat{x}))\subset\hat{W}_{loc}^{u}(\hat{f}^{-n}(\hat{x}))$.
\item for any $\hat{y},\hat{z}\in\hat{W}_{loc}^{u}(\hat{x})$ it holds 
$
dist_{\hat{f}^{-n}(\hat{x})}^{u}(y_{-n},z_{-n})\leq\gamma_{i}e^{-\lambda_{i}n}dist_{\hat{x}}^{u}(y_{0},z_{0})
$
for every $n\in\mathbb{N}$, where $dist_{\hat{x}}^{u}$ is the distance
along $W_{loc}^{u}(\hat{x})$.
\end{enumerate}
\end{enumerate}
\end{proposition}

\subsection{SRB measures}

The notion of SRB measure for endomorphisms depends intrinsically on the 
existence of unstable manifolds, whose geometry and construction are more involving than 
for diffeomorphisms as unstable manifolds may have self intersections (hence do not generate an invariant foliation). 

Let $(X,\mathcal{A},\mu)$ be a probability space and $\mathcal{P}$
be a partition of $X$. We say that $\mathcal{P}$ is a measurable
partition if there is a 
sequence of countable partitions
of $X$, $(\mathcal{P}_{j})_{j\in\mathbb{N}}$, such that
$\mathcal{P}=\bigvee_{j=0}^{\infty}\mathcal{P}_{j}\ \mod0.$
Consider the projection $p:X\rightarrow\mathcal{P}$ that associate
to each $x\in X$ the atom $\mathcal{P}(x)$ that contains $x$.
We say that a subset $\mathcal{Q}\subset\mathcal{P}$ is measurable
if, and only if $p^{-1}(\mathcal{Q})$ is measurable. It is easy to see that the family of all measurable sets of $\mathcal{P}$
is a $\sigma$-algebra of $\mathcal{P}$. 
If $\mu$ is an $f$-invariant probability
measure with at least one positive Lyapunov exponent at $\mu$ almost
every point, a measurable partition $\mathcal{P}$ of $M^{f}$ is
\emph{subordinated to unstable manifolds} if for $\hat{\mu}$
a.e. $\hat{x}\in M^{f}$
(1) $\pi|_{\mathcal{P}(\hat{x})}:\mathcal{P}(\hat{x})\rightarrow\pi(\mathcal{P}(\hat{x}))$
is bijective, and
(2) 
there is a submanifold $W_{\hat{x}}$ with dimension $k(\hat{x})$
in $M$ such that $W_{\hat{x}}\subset W^{u}(\hat{x})$, $\pi(\mathcal{P}(\hat{x}))\subset W_{\hat{x}}$
and $\pi(\mathcal{P}(\hat{x}))$ contains an open neighborhood of $x_{0}$ in $W_{\hat{x}}$,
where $\mathcal{P}(\hat{x})$ denotes the element of the partition $\cP$ that contains $\hat x$ and
$k(\hat x)$ denotes the number of positive Lyapunov exponents at $\hat x$.
Moreover, taking $\tilde{\mu}:=p_{*}\mu$,
,  a
\emph{disintegration of $\mu$ with respect to $\mathcal{P}$} is a family $(\mu_{P})_{P\in\mathcal{P}}$
of probability measures on $X$ so that: 
\begin{enumerate}
\item $\mu_{P}(P)=1$ for $\tilde{\mu}$ a.e. $P\in\mathcal{P}$;
\item for every measurable set $E\subset X$ the map $\mathcal{P}\ni P\mapsto\mu_{P}(E)$ is measurable;
\item $\mu(E)=\int\mu_{P}(E) \, d\tilde{\mu}(P)$ for every measurable set $E\subset X$.
\end{enumerate}
We are now in a position to define the SRB property for invariant measures.
\begin{definition}
\label{def:srbproperty}We say that an $f$-invariant and ergodic probability
measure $\mu$ is an \emph{SRB measure} if  it has at least one positive Lyapunov exponent 
almost everywhere and for every partition $\mathcal{P}$ subordinated
to unstable manifolds one has 
$
\pi_{*}\hat{\mu}_{\mathcal{P}(\hat{x})}\ll Leb_{\,W_{\hat{x}}}
	\;\text{for $\hat{\mu}$ a.e. $\hat{x}\in M^{f}$},
$
where $(\hat{\mu}_{\mathcal{P}(\hat{x})})_{\hat{x}\in M^{f}} $
is a disintegration of $\hat{\mu}$ with respect to $\mathcal{P}$.
\end{definition}

The construction of the SRB measures for these partially hyperbolic attractors follows from careful control of densities of push-forward 
of the Lebesgue measure on disks tangent to the cone field \cite{CV17}.
By (H3), there exists a disk $D$  that is tangent  to the cone field $C$ and so that \eqref{eq:nuexpansion} holds for a positive Lebesgue measure set in $D$ 
and SRB measures arise from the ergodic components of the 
accumulation points of the C\`esaro averages
\begin{equation}\label{eq:munD}
\mu_{n}=\frac{1}{n}\sum_{j=0}^{n-1}f_{*}^{j}Leb_{D}
\end{equation}
The non-uniform expansion on the cone field ensures 
that for every $x\in H\cap D$ 
there are infinitely many values of $n$ (with positive frequency)  so that 
$
\Vert Df^{n}(x)\cdot v\Vert\geq e^{cn}\|v\|
	\quad \text{for {\bf all} }\, v\in C(x). 
$
The previous uniformity is crucial because unstable directions (to be defined
 \emph{a posteriori} at almost every point by means of Oseledets theorem) 
will be contained in the cone field $C$ but cannot be determined a priori, and to prove the SRB measures are hyperbolic measures. 
Moreover, it guarantees that if $\mathcal{D}_{n}$ denotes a suitable family of disks in $D$ that are 
expanded by $f^{n}$ and $\nu$ is an accumulation point of the measures 
\begin{equation}\label{eq:nunD}
\nu_n:=\frac{1}{n}\sum_{j=0}^{n-1}f_{*}^{j}Leb_{\mathcal{D}_{j}}
\end{equation}
then its support is contained in a union of disks 
obtained as accumulations of disks of uniform size $f^{n}(D_n)$, for $D_n\in \mathcal{D}_{n}$, 
as $n\to\infty$.
Due to possible self-intersections of disks $\Delta$ obtained as accumulations of
 the family of disks $f^{n}(\mathcal D_{n})$ tangent to $C$  
a selection procedure on the space of pre-orbits (which also involve 
the lifting of a reference measure to the natural extension) becomes necessary. 
Such a construction involves a careful selection of pre-orbits so that accumulation disks are contained in
unstable manifolds parametrized by elements of the natural extension.
In consequence, any invariant measure supported in these disks will be
hyperbolic since it will have  only positive Lyapunov exponents in the direction complementary to the stable bundle.
 The proof of the SRB property involves a careful and technical choice of partitions adapted to unstable disks on the natural extension that resemble an  inverse Markov tower.


\section{A first volume lemma}\label{sec:volume1}

\subsection{$c$-Cone-hyperbolic times and geometry of disks tangent to $C$} \label{subsec:diskgeometry}

In this subsection we collect some results from \cite{CV17}.
Consider $D$ a disk in $M$ which is tangent to the cone field $C$ and
assume that $D$ satisfies $Leb_{D}(H)>0$, where $H$
is the subset given by hypothesis (H3). For $\vep>0$ and
$x\in M$ denote by
$
T_{x}M(\vep):=\left\{ v\in T_{x}M:\ \|v\|\leq\vep\right\} 
$
the ball of radius $\vep$-centered on $0\in T_{x}M$. Fix $\delta_{0}>0$
such that the exponential map restricted to $T_{x}M(\delta_{0})$
is a diffeomorphism onto its image for every $x\in M$, and set $B(x,\vep):=\exp_{x}(T_{x}M(\vep))$
for every $0<\vep\leq\delta_{0}$.
As $D$ is tangent to the cone field there exists $0<\delta\leq\delta_{0}$ such that if $y\in B(x,\delta)\cap D$
then there exists a unique linear map $A_{x}(y):T_{x}D\rightarrow E_{x}^{s}$
whose graph is parallel to $T_{y}D$. Moreover, by compactness of $M$,
$\delta$ can be taken independent of $x$. Thus one can express the H\"older variation
of the tangent bundle in these local coordinates as follows:
given $C>0$ and $0<\alpha<1$ we say that the tangent bundle $TD$
of $D$ is \emph{$(C,\alpha)$-H\"older} if
$
\|A_{x}(y)\|\leq C\cdot dist_{D}(x,y)^{\alpha},
$
for every $y\in B(x,\delta_{0})\cap D$ and $x\in D$. 
Moreover, the $\alpha$-H\"older constant of the
tangent bundle $TD$ is given by the number:
$
\kappa(TD,\alpha):=\inf\left\{ C>0:\ TD\mbox{ is }(C,\alpha)\mbox{-H\"older}\right\}.
$
Since our assumptions imply that disks tangent to the cone field are preserved under positive iteration by the 
dynamics and assumption (H4) guarantees domination,and a control of the distortion along disks tangent to the cone field
(cf. \cite[Proposition~5.1]{CV17}).
\begin{proposition}\label{prop:iteratesoftgbundleholder}
\cite[Proposition~5.1]{CV17}
There are $\alpha,\beta\in(0,1)$
and $C_{0}>0$ such that if $N$ and $f(N)$ are submanifold of $M$ 
tangent to the cone field $C$ 
then $\kappa(Df(N),\alpha)\leq\beta\cdot\kappa(TN,\alpha)+C_{0}$.
In consequence, there is $L_{1}>0$ (depending only on $f$) such that if $f^j(N)$ is a submanifold of $M$
 tangent to $C$ for all $0\le j \le k$ 
then the map
$$
\begin{array}{rccc}
J_k: & f^k(N) & \rightarrow & \mathbb{R}\\
     & x      & \mapsto     & \log\left| \det( Df(x)|_{T_xf^k(N)})\right|,
\end{array}
$$is $(L_{1},\alpha)$-H\"older continuous. 
\end{proposition}
Hypothesis (H3) implies that there exists a disk $D\subset U$ tangent to $C$
and a positive Lebesgue measure subset $H\subset D$ 
satisfying
\begin{equation}
\limsup_{n\to\infty}\frac{1}{n}\sum_{j=0}^{n-1}\log\|(Df(f^{j}(x))|_{C(f^{j}(x))})^{-1}\|\leq-2c<0\label{eq:limsupnegative}
\end{equation}
for every $x\in H$. Pliss' lemma ensures that there exists $\theta>0$ 
(that depends only on $f$ and $c$) so that the following holds: for every $x\in H$ and $n\ge 1$ large there exists
a sequence $1\leq n_{1}(x)<\cdots<n_{l}(x)\leq n$ of $c$-cone-hyperbolic times for $x$ with $l\geq\theta n.$\label{lem:densposth}

\begin{definition}\label{def:non-lacunar}
Given $x\in \Lambda$ with infinitely many cone-hyperbolic times, let  $n_i=n_i(x)$ denote the $i$th cone hyperbolic time for $x$. 
We say that the sequence $(n_i)_{i\ge 1}$ is \emph{non-lacunar} if $n_{i+1}/n_i \to 1$ as $i\to\infty$.
\end{definition}

\begin{remark}\label{rmk:nonlacLeb}
We observe that the non-lacunarity of a given sequence of positive integers
requires the gaps between consecutive elements to be small in comparison with the growth of the elements in the sequence. 
Nevertheless, sequences of hyperbolic times are commonly non-lacunar, and this fact is a consequence of 
the integrability of the first hyperbolic time map (see e.g. \cite[Lemma~3.9]{Var12}).
\end{remark}

\begin{remark}
A simple computation shows that 
$n_1(f^{n_{i}(x)}(x))=n_{i+1}(x)-n_{i}(x)$
and that, for every $\vep>0$,  $\Leb( x\in U \colon n_{i+1}(x)-n_{i}(x) > \vep n_i(x))$ coincides with 
$\Leb( x\in U \colon n_1(f^{n_{i}(x)}(x))>  \vep n_{i}(x))$. 
Consequently, the condition
\begin{equation}\label{eq:suff-condition}
\forall \vep>0\colon  \quad \sum_{k\ge 1} \Leb( x\in U \colon n_1(f^{k}(x))>  \vep k) < \infty
\end{equation}
together with Borel-Cantelli implies that the sequence of cone-hyperbolic times is Lebesgue almost everywhere non-lacunar.
Ara\'ujo and Pac\'ifico \cite{AP06}, presentd large deviations upper bounds for the Lebesgue measure 
using an alternative method of covering at instants of hyperbolicity. 
\end{remark}

Given a submanifold $D \subset M$ we denote by $dist_{D}$ the
Riemannian distance on $D$  then $dist_{M}\leq dist_{D}$.

\begin{proposition}\label{prop:contraidist}\cite[Proposition~5.2]{CV17}
Let $D$ be a $C^{1}$ disk embedded in $U$ that is tangent to $C$. There is $\delta>0$ such that if $n$ is a 
$c$-cone-hyperbolic time for $x\in D$ and
$d(x,\partial D)>\delta$ 
then there is an open neighborhood $D(x,n,\delta)$ of $x$ in $D$ diffeomorphic 
to a $C^1$-disk $\Delta(f^{n}(x),\delta)$ by $f^{n}$. Moreover 
$$
dist_{f^{n-k}(D(x,n,\delta))}(f^{n-k}(x),f^{n-k}(y))\leq e^{-\frac{c}{2}k}dist_{\Delta(f^{n}(x),\delta)}(f^{n}(x),f^{n}(y))
$$
for every $y\in D(x,n,\delta)$ and $1\leq k\leq n.$
\end{proposition}

We observe the constant $\delta>0$ in Proposition~\ref{prop:contraidist} is independent of $x\in M$.
The disks $D(x,n,\delta)\subset D$ are called \emph{hyperbolic pre-disks} and their 
images $\Delta(f^{n}(x),\delta)$ are called \emph{$n$-hyperbolic disks}.
We need the following bounded distortion property (cf. \cite[Proposition~5.2]{CV17}) for the Jacobian of $f$ along disks tangent to
the invariant cone field.

\begin{proposition}
\label{prop.dist. volume ABV}
Suppose $D$ is a $C^{1}$ disk embedded in $U$, $x\in D$ and let $n$ be
a $c$-cone hyperbolic time for $x$. There is $C_{1}>0$ such that
\[
C_{1}^{-1}\leq\dfrac{|det\ Df^{n}(y)|_{T_{y}D(x,n,\delta)}|}{|det\ Df^{n}(z)|_{T_{z}D(x,n,\delta)}|}\leq C_{1}
\]
for every $y,z\in D(x,n,\delta)$, where $D(x,n,\delta)$ is the hyperbolic
pre-disk at $x$.\label{prop:distorcionhyptime}
\end{proposition}

\subsection{The first volume lemma}\label{subsec.volume lemma 1}

In this subsection we restate and prove a volume lemma for the Lebesgue measure of dynamic balls on the ambient space.
The volume lemma follows from a Fubini-like argument together with two bounded distortion lemmas, one obtained as 
consequence of backward contraction property for disks tangent to the cone field at $c$-cone-hyperbolic times, and the other that
requires the control of distortion for all subspaces of maximal dimension contained in the cone field.

Recall that the cone field $C$ has constant dimension $d$.  Given $x \in M$ and two subspaces $F_1,F_2$ of dimension $d$ 
contained in the cone $C(x)$, 
there exist unique linear maps $L_{F_2}:F_1\rightarrow E^s_{x}$ and 
$L_{F_1}:F_2\rightarrow E^s_{x}$ such that 
\begin{equation}\label{defF2}F_2=\graph(L_{F_2}):=\{ v + L_{F_2}(v) \in F_1\oplus E_x^s \colon v\in F_1\},
\end{equation} 
and $F_1=\graph(L_{F_1})$. We measure the distance between two subspaces $F_1$ and $F_2$ on $T_x M$ by
$$\theta_x(F_1,F_2):=\max\{ \|L_{F_1}\|, \|L_{F_2}\|\}.$$
Note that $\theta_x(F_1,F_2)=0$ if and only if $F_1$ and $F_2$ coincide and this is roughly a projective metric.

\begin{remark}
Since $f$ is a local diffeomorphism, given $x\in M$, two subspaces $F_1, F_2 \subset T_x M$ of dimension $d$ and the linear map $L_{F_2}: F_1 \to E^s_x$ such that $F_2=\graph(L_{F_2})$, the subspaces $Df(x) F_1$ and $Df(x) F_2$
also have dimension $d$. Moreover, $Df(x) F_2$ is obtained as the graph of a unique linear map $L^1_{F_2}: Df(x) F_1 \to E^s_{f(x)}$. In addition, by \eqref{defF2} and the $Df$-invariance of the stable subbundle we get 
$$
Df(x) [v + L_{F_2}(v)] = Df(x) v + Df(x) L_{F_2}(v) \in \big[Df(x) F_1 \oplus E^s_{f(x)}\big] \cap \graph(L^1_{F_2}),
$$
and, consequently, $Df(x)L_{F_2}(v)=L_{F_2}^1 (Df(x)v)$ for each $v \in F_2$ (see Figure~1 below). 
\begin{figure}[htb]
\begin{center}
        \includegraphics[scale = 0.4]{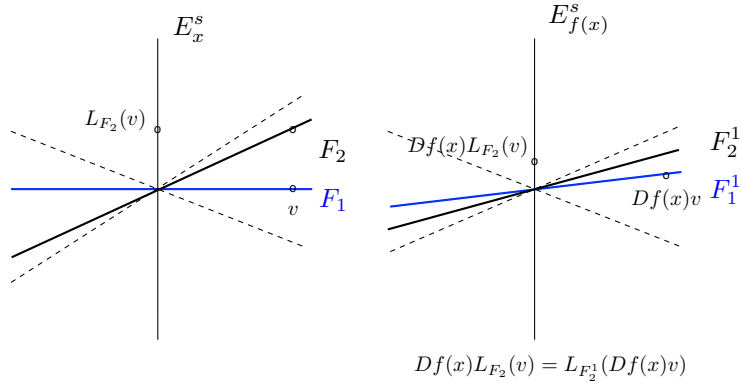}
        \caption{Iteration of subspaces that are graphs}
\end{center}
\end{figure}
By induction, we get that 
\begin{equation}\label{eq:invarianceLs}
Df^i(x)L_{F_2}(v)=L_{F_2}^i (Df^i(x)v) \quad \text{for every $v \in F_2$ and every $i\ge 1$},
\end{equation}
where $F_2^i=Df^i(x) F_2=\graph(L_{F_2}^i)$, $L^i_{F_2}: F_1^i \to E^s_{f^i(x)}$ and $F_1^i:=Df^i(x) F_1$.
\end{remark}

\begin{lemma}\label{lem. distorcao volume 1}
There exists $\kappa \in (0,1)$ such that for every $n\ge 1$, 
$$\theta_{f^i(x)}(Df^i(x)F_1,Df^i(x)F_2)\leq \kappa^i\cdot \theta_x(F_1,F_2)
\quad \text{for every $1\le i \le n$}.$$
\end{lemma}

\begin{proof}
Given $n\ge 1$, $v\in F_1$ and $v_i=Df^i(x)v$ for every $0\le i \le n$. 
The invariance condition~\eqref{eq:invarianceLs} and the domination condition (H4) 
	\begin{eqnarray*}
		\frac{\|L_{F_2^i}v_i\|}{\|v_i\|}
		=\frac{\|Df^i(x)L_{F_2}v\|}{\|Df^i(x)v\|}
		\le \lambda^i \frac{\|L_{F_2}v\|}{\|v\|}
		\leq \lambda^i \theta_x(F_1,F_2) 
	\end{eqnarray*}
for every $0\le i \le n$. 
Since estimates for $\|L_{F_1^i}\|$ are analogous the lemma follows by taking $\kappa=\lambda$.
\end{proof}

\begin{proposition}\label{lem. distorcao volume 2}
If $n\ge 1$, $d>1$, $F_1,F_2\subset T_xM$ are subspaces of dimension $d$ contained in $C(x)$ and $\theta_x(F_1,F_2)<\frac{1}{d^2-1}$ then there exists
$K>0$  such that
$$
\left|\frac{|det(Df^n(x)|_{F_2})|}{|det(Df^n(x)|_{F_1})|} - 1\right|\leq  K \, \frac{\theta_x(F_1,F_2)}{1-(d^2-1)\theta_x(F_1,F_2)}.
$$
\end{proposition}

\begin{proof}
Assume $d>1$. Given a orthonormal basis  $e_1,...,e_d$ of $F_1$ we have that $\|e_1\wedge...\wedge e_d\|=1$ and
	\begin{equation}\label{eq.defDet}	
	|det(Df^n(x)|_{F_1})|=\|Df^n(x)e_1\wedge...\wedge Df^n(x)e_d\|.
	\end{equation}
	We proceed to estimate $|det(Df^n(x)|_{F_2})|$.
	Note that for any pair of orthonormal vectors $e_s, e_r$ on $F_1$  we have
	\begin{align*}
	&|\,\|Df^n(x)(e_r+  L_{F_2}e_r)\wedge Df^n(x)(e_s+L_{F_2}e_s)\|-\|Df^n(x)e_r\wedge Df^n(x)e_s\|\,| \\
		&\leq \|Df^n(x)e_r\wedge Df^n(x)L_{F_2}e_s + Df^n(x)L_{F_2}e_r\wedge Df^n(x)e_s
		+ Df^n(x)L_{F_2}e_r\wedge Df^n(x)L_{F_2}e_s\|\\
		&=\| Df^n(x)e_r\wedge L^n_{F_2} Df^n(x)e_s +  L^n_{F_2}Df^n(x)e_r \wedge Df^n(x)e_s
		+ L^n_{F_2}Df^n(x)e_r \wedge L^n_{F_2}Df^n(x)e_s\|\\
		&\leq 3\|L^n_{F_2}\| \|Df^n(x)e_r\wedge Df^n(x)e_s\|.
	\end{align*}
Using the previous argument recursively, we get
\begin{align*}
	\!\!\!\!\!\!\!\!\!\!\!|\,\|Df^n(x)(e_1+L_{F_2}e_1)\wedge...\wedge & Df^n(x)(e_d+L_{F_2}e_d)\|-\|Df^n(x)e_1\wedge...\wedge Df^n(x)e_d\|\,|\\
	&\leq (d^2-1)\|L^n_{F_2}\|\|Df^n(x)e_1\wedge...\wedge Df^n(x)e_d\|,
\end{align*} 
which implies that
\begin{align}\label{eq. est.det.1}
&\|Df^n(x)(e_1+L_{F_2}e_1)\wedge...\wedge Df^n(x)(e_d+L_{F_2}e_d)\|\\\nonumber
&\leq [1+(d^2-1)\|L^n_{F_2}\|] \,\|Df^n(x)e_1\wedge...\wedge Df^n(x)e_d\|
\end{align}
and
\begin{eqnarray}\label{eq. est.det.2}
&\|Df^n(x)(e_1+L_{F_2}e_1)\wedge...\wedge Df^n(x)(e_d+L_{F_2}e_d)\|\\ \nonumber
&\geq [1-(d^2-1)\|L^n_{F_2}\|] \,\|Df^n(x)e_1\wedge...\wedge Df^n(x)e_d\|.
\end{eqnarray}
Proceeding as above we also obtain that 
\begin{eqnarray*}
|\,\|(e_1+L_{F_2}e_1)\wedge...\wedge (e_d+L_{F_2}e_d)\|-\|e_1\wedge...\wedge e_d\|\,|\leq (d^2-1)\theta_x(F_1,F_2).
\end{eqnarray*}
or, equivalently,
\begin{eqnarray}\label{eq. est.prod.ext.2}
1-(d^2-1)\theta_x(F_1,F_2)\le \|(e_1+L_{F_2}e_1)\wedge...\wedge (e_d+L_{F_2}e_d)\| \leq 1+(d^2-1)\theta_x(F_1,F_2)
\end{eqnarray}
By Lemma~\ref{lem. distorcao volume 1}, together with \eqref{eq.defDet}, \eqref{eq. est.det.1} and \eqref{eq. est.prod.ext.2}, 
we conclude that 
	\begin{eqnarray*}
		|det(Df^n(x)|_{F_2})|&=&\frac{\|Df^n(x)(e_1+L_{F_2}e_1)\wedge...\wedge Df^n(x)(e_d+L_{F_2}e_d)\|}{	\|(e_1+L_{F_2}e_1)\wedge...\wedge (e_d+L_{F_2}e_d)\|}\\
		&\leq& \frac{1+(d^2-1)\|L^n_{F_2}\|}{1-(d^2-1)\theta_x(F_1,F_2)} \cdot\|Df^n(x)e_1\wedge...\wedge Df^n(x)e_d\|  \\
	&\leq& \frac{1+(d^2-1) \theta_{f^n(x)}(F_1^n,F_2^n) }{1-(d^2-1)\theta_x(F_1,F_2)} \cdot\|Df^n(x)e_1\wedge...\wedge Df^n(x)e_d\|,
	\end{eqnarray*}
and so
\begin{align*}
	\frac{|det(Df^n(x)|_{F_2})|}{|det(Df^n(x)|_{F_1})|}
	&\leq \frac{1+(d^2-1) \theta_{f^n(x)}(F^n_1,F^n_2) }{1-(d^2-1)\theta_x(F_1,F_2)} 
	 = 1+ \frac{(d^2-1)[\theta_x(F_1,F_2) + \theta_{f^n(x)}(F^n_1,F^n_2) ] }{1-(d^2-1)\theta_x(F_1,F_2)} \\
	& \le 1+ \frac{(d^2-1)[\theta_x(F_1,F_2) + \theta_{f^n(x)}(F_1^n,F_2^n) ] }{1-(d^2-1)\theta_x(F_1,F_2)}\\ 
	& \le 1+ \frac{(d^2-1)(1+\kappa^n)}{1-(d^2-1)\theta_x(F_1,F_2)}  \theta_x(F_1,F_2) \\
		& \le 1+ 2(d^2-1)\frac{\theta_x(F_1,F_2)}{1-(d^2-1)\theta_x(F_1,F_2)}  
\end{align*}
for every $n\ge 1$.
Analogously, using \eqref{eq.defDet}, \eqref{eq. est.det.2} and \eqref{eq. est.prod.ext.2},
\begin{align*}
	\frac{|det(Df^n(x)|_{F_2})|}{|det(Df^n(x)|_{F_1})|}
	&\geq \frac{1-(d^2-1) \theta_{f^n(x)}(F_1^n,F_2^n) }{1+(d^2-1)\theta_x(F_1,F_2)}\\
	& = 1- \frac{(d^2-1) \theta_{f^n(x)}(F_1^n,F_2^n) +(d^2-1)\theta_x(F_1,F_2)}{1+(d^2-1)\theta_x(F_1,F_2)} \\
	& \ge 1- \frac{(d^2-1) (1+\kappa^n) \theta_x(F_1,F_2)}{1+(d^2-1)\theta_x(F_1,F_2)}
	\ge 1- \frac{2(d^2-1)  \theta_x(F_1,F_2)}{1+(d^2-1)\theta_x(F_1,F_2)}.
\end{align*}
Taking $K=2(d^2-1)$ the lemma follows. 
\end{proof}

\begin{remark}\label{rmk:d1}
If $d=1$ a simplification of the previous argument yields that if $e_1\in F_1$ and $\|e_1\|=1$ then
$$
\frac{|det(Df^n(x)|_{F_2})|}{|det(Df^n(x)|_{F_1})|} = \frac{\|Df^n(x)(e_1+L_{F_2}e_1)\|}{	\|e_1+L_{F_2}e_1\|} 
	\le \frac{1+\|Df^n(x)L_{F_2}e_1\|}{	1- \theta_x(F_1,F_2)} \le \frac{1+\theta_{f^n(x)}(F_1^n,F_2^n)}{	1- \theta_x(F_1,F_2)}
$$
and the lower bound follows analogously.
\end{remark}

Using Lemma \ref{lem. distorcao volume 1} and Proposition~\ref{lem. distorcao volume 2}, we deduce the following bounded
distortion property:

\begin{corollary}\label{cor:subspaces}
There exists $K_0>0$ 
so that for every  $n \ge 1$ 
\begin{equation}\label{eq:dist1}
\frac1{K_0}\le \frac{|det(Df^n(x)|_{F_2})|}{|det(Df^n(x)|_{F_1})|} \le K_0
\end{equation}
for {\bf any} subspaces  $F_1,F_2$ of dimension d contained in the cone $C(x)$. 
\end{corollary}

\begin{proof}
Assumptions (H1)-(H4) ensure that the angle 
between vectors in $E^s$ and vectors contained in the cone field $C(\cdot)$ is uniformly bounded away from zero on the attractor $\Lambda$. 
Hence $\theta_{x}(F_1,F_2)$  
is uniformly bounded, for every $x\in \Lambda$ and any pair of subspaces 
$F_1,F_2$ contained in $C(x)$.  This, together with Proposition~\ref{lem. distorcao volume 2} and Remark~\ref{rmk:d1} 
implies that ~\eqref{eq:dist1} holds with $K_0 = 1+ 2 K \max_{x\in \Lambda} \max_{F_1,F_2 \subset C(x)} \theta_x(F_1,F_2)$
for any subspaces  $F_1,F_2$ of dimension d contained in the cone $C(x)$ so that $\theta_x(F_1,F_2)<\min\{\frac{1}{2(d^2-1)},\frac12\}$. 

Now, let $F_1,F_2$ be arbitrary subspaces of dimension d contained in $C(x)$. Recall that $F_2$ is the graph of a linear map 
$L_{F_2} : F_1 \to E^s_x$. The Grassmanian  
$Grass_d(\mathbb R^{\dim M})$, that is, the set of all $d$-dimensional subspaces in $\mathbb R^{\dim M}$, is compact 
(see e.g. \cite{Sternberg}). Therefore, there exists $N=N(d)\ge 1$ so that $Grass_d(\mathbb R^{\dim M})$ admits a 
finite cover by $N$ balls of radius $\frac{1}{2(d^2-1)}$.  Using that $M$ is compact and that the fiber bundle $TM$ is locally isomorphic to 
$U\times \mathbb R^{\dim M}$ ($U\subset M$) we have the following consequence: there exists $1\le r \le N$ and a family $(E_i)_{1\le i \le r}$ of subspaces of dimension $d$ contained in the cone such that $E_1=F_1$, $E_r=F_2$ and $\theta_x(E_i,E_{i+1})<\frac{1}{2(d^2-1)}$ for every $1\le i \le r-1$.
In particular \eqref{eq:dist1} holds with $K_0=(1+ 2 K \max_{x\in \Lambda} \max_{F_1,F_2 \subset C(x)} \theta_x(F_1,F_2))^N$.
This proves the corollary.
\end{proof}

We will need the following:
 
\begin{corollary}\label{cor:hyp-volume}
Let $\nu$ be an $f$-invariant probability measure supported on $\Lambda$. Then, for $\nu$-almost every $x$ and for {\bf any} subspace  $F_1 \subset C(x)$ of dimension d
$$
\Gamma_\nu(x)=\lim_{\substack{n\to\infty}}  \frac1n \log |\det Df^n(x)\!\mid_{F_1}| 
	= \sum_{\lambda_i(\nu,x)>0} \; \lambda_i(\nu,x),
$$
where $\lambda_1(\nu,x) \le \lambda_2(\nu,x)\le \dots \le \lambda_{d}(\nu,x)$ are the $\nu$-almost everywhere defined 
Lyapunov exponents of $f$ with respect to $\nu$ along the cone direction. 
\end{corollary}

\begin{proof}
Let $\nu$ be $f$-invariant probability measure supported on $\Lambda$ and let $\hat\nu$ be the unique $\hat f$-invariant probability
measure such that $\pi_*\hat\nu=\nu$. 
By Oseledets theorem there exists a almost everywhere defined
$D\hat f$-invariant splitting $T_{\pi(\hat{x})}M =E_{1}(\hat{x})\oplus\dots\oplus E_{r(\hat{x})}(\hat{x})$ so that  
$\lim_{n\to\pm\infty}\frac{1}{n}\log\|D\hat{f}^{n}(\hat{x}) u\|=\lambda_{i}(\hat{x})$
for every $u\in E_{i}(\hat{x})\backslash\left\{ 0\right\} $
and for every $1\leq i\leq r(\hat{x})$ (cf. Proposition~\ref{prop:existsrb})). 
However, by partially hyperbolicity (recall assumptions (H1) and (H3)) $E^s$ is a $Df$-invariant subbundle 
and $C(\cdot)$ is a $Df$-invariant cone field.
This fact, together with the latter, ensures that for $\hat\nu$-almost every $\hat x$ there exists 
$1<s(\hat x) <r(\hat x)$ so that
$$
E^s_{\pi(\hat x)}=E_{1}(\hat{x}) \oplus \dots \oplus E_{s(\hat{x})}(\hat{x})
\;\; \text{and} \;\;
T_{\pi(\hat{x})}M = E^s_{\pi(\hat x)} \oplus E_{s(\hat{x})+1}(\hat{x}) \oplus\dots\oplus E_{r(\hat{x})}(\hat{x}),
$$
and the $D\hat f$-invariant subspace $\hat F_{\pi(\hat x)}:=E_{s(\hat{x})+1}(\hat{x}) \oplus\dots\oplus E_{r(\hat{x})}(\hat{x})$
is contained in the cone $C(\pi(\hat x))$.

Now, fix an arbitrary subspace $F_1 \subset C(x)$ of dimension d.
By definition of $\Gamma_n$ (recall~\eqref{def:Gamma}) together with Corollary~\ref{cor:subspaces} it follows that
\begin{equation}\label{Gammaprox}
|\det Df^n(x)\!\mid_{F_1}| \le \Gamma_n(x):=\max_{\tilde F \subset C(x)} |\det Df^n(x)\!\mid_{\tilde F}|
	\le K_0 |\det Df^n(x)\!\mid_{F_1}|.
\end{equation}
Thus, the limit $\lim_{i\to\infty}  \frac1{n_i} \log \det |Df^{n_i}(x)\!\mid_{F_1}|$, taken along the sequence $(n_i)_i$ of $c$-cone-hyperbolic times for $x$, exists and coincides with $\Gamma_\nu(x)$
for $\nu$-a.e. $x$.
Moreover, using \eqref{Gammaprox} once more, the fact that $(\Gamma_n)_n$ is sub-multiplicative and 
Kingman's subadditive ergodic theorem then $\Gamma_\nu(x)=\lim_{\substack{n\to\infty}}  \frac1n \log |\det Df^n(x)\!\mid_{F_1}|$.
Finally, the second equality in the corollary follows from the previous argument 
and the fact that if $x=\pi(\hat x)$ 
then
$
K_0^{-1} |\det Df^n(x)\!\mid_{F_1}| \le |\det Df^n(x)\!\mid_{\hat F_{x}}|
	\le K_0 |\det Df^n(x)\!\mid_{F_1}|.
$
This proves the corollary.
\end{proof}

\begin{remark}\label{rmk:open}
Given $c>0$ and $\delta=\delta(c)>0$ as in Proposition~\ref{prop:contraidist}, there exists $\delta_1=\delta_1(c) \in (0,\delta)$
so that if $n$ is a $c$-cone-hyperbolic time for $x$ then
$n$ is a $\frac{c}2$-cone-hyperbolic time for every $y\in B(x,n,\delta_1)$
 (see e.g. Lemma 5.3 and Proposition 5.7 in \cite{CV17}).
\end{remark}

Let $\Leb$ denote the Lebesgue measure on the compact Riemannian manifold $M$. The next proposition ensures that the Lebesgue 
measure satisfies a volume lemma (at cone-hyperbolic times) expressed in terms of the Jacobian along of any subspace of full dimension in
 the cone field. More precisely:

\begin{proposition}\label{le:volume1}
 Let $\delta=\delta(\frac{c}2)>0$ be given by Proposition~\ref{prop:contraidist}.
Given $0<\vep<\delta$ there exists $K(\vep)>0$ so that if $n$ is a $c$-cone-hyperbolic time for $x$ and $F_1$ is {\bf any} 
subspace of dimension $d$ contained in $C(x)$ then
\begin{equation}\label{eq:vol1}
K(\vep)^{-1} e^{-\log |\det Df^n(x)\mid_{F_1}|} \le \Leb(B(x,n,\vep)) \le K(\vep) e^{-\log |\det Df^n(x)\mid_{F_1}|}.
\end{equation}
\end{proposition}

\begin{proof}
Given $c>0$ let $0<\delta_1<\delta=\delta(\frac{c}2)$ be given by Proposition~\ref{prop:contraidist} and Remark~\ref{rmk:open}.
Assume that $n$ is a $c$-cone-hyperbolic time for $x$ and take $0<\varepsilon\ll \delta_1$. 
By Remark~\ref{rmk:open}, $n$ is a $\frac{c}2$-cone hyperbolic time
for every $y\in B(x,n,\delta_1)$ and, consequently, if $ D$ is a $C^1$-disk tangent to $C(x)$ passing through $y$ 
and $d(y,\partial  D)>\vep$  then there is an open neighborhood $ D(y,n,\vep)$ of $y$ in $ D$ diffeomorphic 
to a $C^1$-disk $ \Delta(f^{n}(y),\vep)$ by $f^{n}$ and such that 
$$
dist_{f^{n-k}( D(y,n,\vep))}(f^{n-k}(z),f^{n-k}(y))\leq e^{-\frac{c}{2}k}dist_{ \Delta(f^{n}(y),\vep)}(f^{n}(z),f^{n}(y)),
$$
for every $z\in  D(y,n,\vep)$ and $1\leq k\leq n.$
Since the curvature of disks tangent to $C(x)$ is bounded, we may reduce $\delta_1$ is necessary to get that $dist_{M} \le dist_{\tilde D} \le 2dist_M$ for every $C^1$-disk $\tilde D$ of diameter $\delta_1$ tangent to the cone field.

Let $D$ be a  $C^1$-disk of diameter $\delta_1$ tangent to $C(x)$ passing through $x$ and let $D(x,n,\vep)$ be as above.  
Then it is not hard to check that 
$$
D \cap B(x,n,\frac\vep2) 
	\subset D(x,n,\vep) 
	\subset D\cap B(x,n,\vep). 
$$
Moreover, as the local stable manifolds are exponentially contracted by iteration of the dynamics, we have that
$W^s_{loc}(x) \cap B(x,\vep) = W^s_{loc}(x) \cap B(x,n,\vep)$. 
Set $W:=W^s_{loc}(x) \cap B(x,\vep)$,
consider $\hat W=\exp_x^{-1}(W)$,
$\hat D_x:=\exp_x^{-1}(D)$,  and for every $y\in \hat W$ let $\hat D_y$ be the $C^1$-submanifold passing through $y$ and 
obtained by parallel transport from $\hat D_x$ (see Figure~2). 
\begin{figure}[htb]
\begin{center}
        \includegraphics[scale = 0.35]{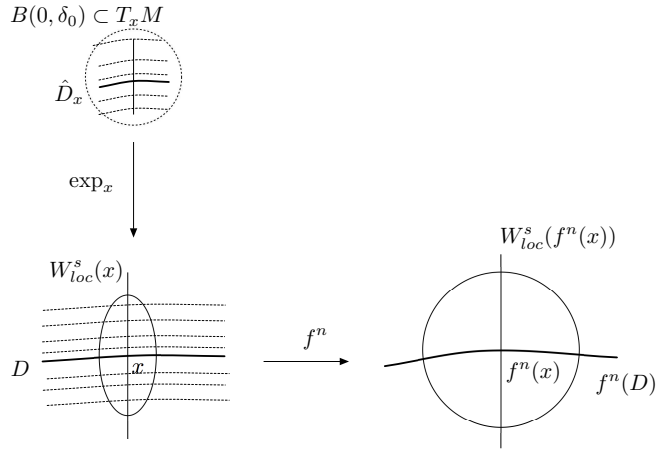}
        \caption{Construction of a foliation of the dynamic ball by disks tangent to the cone field}
\end{center}
\end{figure}
In particular $(\exp_x(\hat D_y))_{y\in \hat W}$ is a foliation by $C^1$-submanifolds $(\gamma_y)_{y\in W}$ tangent to the cone field $C$ that cover $B(x,n,\vep)$. 
By Fubini-Tonelli Theorem, we have that
$
\Leb(B(x,n,\vep))=\int_{W} \Leb_{{\gamma_y}}(B(x,n,\vep)\cap \gamma_y) \;dy.
$ 
Using change of variables and Proposition~\ref{prop.dist. volume ABV}, we conclude that, for every $y\in W$,
	\begin{eqnarray*}
	\Leb(f^n(B(x,n,\vep)\cap \gamma_y))
	&=& \int_{B(x,n,\vep)\cap \, \gamma_y}  |\det Df^n(z) \mid_{T_z \gamma_y}| \; d\Leb_{{\gamma_y}}(z) \label{eq:dynball1d0}\\
	&\le& C_1 \Leb_{{\gamma_y}}(B(x,n,\vep)\cap \gamma_y) \; |\det Df^n(y) \mid_{T_y \gamma_y}| \nonumber
	\end{eqnarray*}
(the lower bound is analogous).
Since $n$ is a $c$-cone-hyperbolic time then $f^n(B(x,n,\vep)\cap \gamma_y)$ contains 
a disk of uniform radius $\vep$ in $\gamma_y$, and the left hand side above is bounded away from zero by some constant $C_2(\vep)>0$. By Corollary~\ref{cor:subspaces}, 
\begin{align}
 \Leb_{{\gamma_y}}(B(x,n,\vep)\cap \gamma_y) 
 	& \ge C_1^{-1} C_2(\vep) |\det Df^n(y) \mid_{T_y \gamma_y}|^{-1} \nonumber\\
	&  \ge K_0^{-1} C_1^{-1} C_2(\vep)  |\det Df^n(y) \mid_{F_1}|^{-1}  \label{eq:dynball1d}
\end{align}
for all $y\in W$ (by some abuse of notation $F_1$ stands for the subspace of $T_y M$ obtained from $F_1\subset T_x M$ by parallel transport). 
Since $\Leb(f^n(B(x,n,\vep)\cap \gamma_y))$ is bounded above by a uniform constant $C_3(\vep)>0$ a similar reasoning implies that
\begin{align}
 \Leb_{{\gamma_y}}(B(x,n,\vep)\cap \gamma_y) 
	&  \le K_0 C_1 C_3(\vep)  |\det Df^n(y) \mid_{F_1}|^{-1}  \label{eq:dynball1d2},
\end{align}
 Now, since the distance between any pair of points in $W$ is exponentially contracted by forward iteration 
and $W\ni y\mapsto |\det Df(y) \mid_{F_1}|$ is H\"older continuous then 
there exists $K_1>0$ so that $K_1^{-1}|\det Df^n(y) \mid_{F_1}| \le |\det Df^n(x) \mid_{F_1}| \le K_1 |\det Df^n(y) \mid_{F_1}|$
for any $x,y\in W$ and $n\ge 1$. Therefore
$$
\Leb(B(x,n,\vep)) = \int  \Leb_{{\gamma_y}}(B(x,n,\vep)\cap \gamma_y) \;dy
	 \ge K(\vep)^{-1} e^{-\log|\det Df^n(x) \mid_{F_1}|},
$$
for every $n\ge1$, where $K(\vep)=K_0 K_1 C_1 \max\{C_2(\vep)^{-1},C_3(\vep)\}$. Since the proof of the upper bound is entirely analogous using
\eqref{eq:dynball1d2} instead of \eqref{eq:dynball1d}, this completes the proof of the proposition.
\end{proof}

\section{A second volume lemma}\label{sec:volume2}

\subsection{The construction of the SRB measures}\label{sec:SRB}

We first recall some of the arguments in the construction of the SRB measures, referring the reader
to \cite{CV17} for full details and proofs.
Fix $\delta>0$ given by Proposition \ref{prop:contraidist} and
$H=H(c) \subset U$ be the set of points with infinitely many $c$-cone hyperbolic times
and such that $Leb_{D}(H)>0$, guaranteed by hypothesis (H3). Reducing $\delta$, if necessary,
we can assume that 
$
B:=B(D,c,\delta):=\left\{ x\in D\cap H:\ dist_{D}(x,\partial D)\geq\delta\right\} 
$
satisfies $Leb_{D}(B)>0$. For each $n\in\mathbb{N}$ define
$H_{n}:=\left\{ x\in B:\ n\mbox{ is a }c\mbox{-cone hyperbolic time for }x\right\} .$
By \cite[Section 5]{CV17}, 
there exists $\tau>0$ and for each $n\in\mathbb{N}$
there is a finite set $H_{n}^{*} \subset H_{n}$ such that the hyperbolic
pre-disks $(D(x,n,\delta))_{x\in H_{n}^{*}}$ are contained in $D$, pairwise disjoint and their union 
$
\mathscr{D}_{n}:=\bigcup_{x\in H_{n}^{*}}D(x,n,\delta)
$
 satisfies $\Leb_{D}(\mathscr{D}_{n}\cap H_{n})\geq\tau Leb_{D}(H_{n})$.
The sequence of probability measures $(\nu_n)_n$ given by ~\eqref{eq:nunD}
have non-trivial accumulation points $\nu$ in the weak$^*$ topology.
Each measure $\nu_{n}$ (recall ~\eqref{eq:nunD}) is supported on the union of disks 
$\cup_{j=0}^{n-1}\Delta_{j}$, with $\Delta_{j}:=\cup_{x\in H_{j}^{*}} \Delta(f^{j}x,\delta)$ 
and the support of any accumulation point  $\nu$ of $(\nu_{n})_{n\in\mathbb{N}}$ is contained 
in a union of unstable manifolds of uniform size.

\begin{proposition} \cite{CV17}
\label{lem:suppdisk} Given $y\in\supp(\nu)$ exist $\hat{x}=(x_{-n})_{n\in\mathbb{N}}\in M^{f}$
such that $y$ belongs to a disk $\Delta(\hat{x})$ of
radius $\delta>0$ accumulated by $n_{j}$-hyperbolic disks $\Delta(f^{n_{j}}(x),\delta)$
with $j\to\infty$, and satisfying that
\begin{enumerate}
\item the inverse branch $f_{x_{-n}}^{-n}$ is well defined in $\Delta(\hat{x})$
for every $n\geq0$;
\item if for each $y\in\Delta(\hat{x})$ one takes $y_{-n}:=f_{x_{-n}}^{-n}(y)$ then
$
dist_{M}(y_{-n},x_{-n})\leq e^{-\frac{c}{2}\cdot n}\delta\mbox{ for every }n\geq0.
$
\end{enumerate}
\end{proposition}

By Proposition \ref{lem:suppdisk}, given $y\in\supp(\nu)$ there
is a disk $\Delta(\hat{x})$ which is contractive along
the pre-orbit $\hat{x}\in M^{f}$. 
We proceed to prove $\hat{\nu}$ is supported on a union of unstable manifolds. 
For that, if $\Delta(\hat{x})$ ($\hat{x}\in\hat{H}_{\infty}$) is a disk given by Proposition \ref{lem:suppdisk} 
then there exists a sequence of points
$(z_{j_{k}})_{k\in\mathbb{N}}$ satisfying  $z_{j_{k}}\in H_{j_{k}}^{*}$
for every $k\in\mathbb{N}$ 
such that $\lim_{k\to\infty}\Delta(f^{j_{k}}(z_{j_{k}}),\delta)=\Delta(\hat{x})$.
Since $\lim_{k\to\infty}f^{j_{k}-n}(z_{j_{k}})=x_{-n}$ for every $n\in\mathbb{N}$ (where the limit is taken over $j_{k}\geq n$)
then it is natural to define:
\begin{align}
\hat{\Delta}(\hat{x}) & :=\left\{ \hat{y}\in M^{f}:\ \hat{y}=\lim_{s\to\infty}\hat{y}_{j_{k_{s}}},\text{ for some subsequence }(\hat y_{j_{k_{s}}})_{s\in\mathbb{N}}\right. \nonumber \\ 
	& \left.\text{where }\hat{y}_{j_{k_{s}}}\in\hat{f}^{j_{k_{s}}}(\pi^{-1}(D(z_{j_{k_{s}}},j_{k_{s}},\delta)))\text{ and }(z_{j_{k_{s}}})_{\!s\in\mathbb{N}}\!\subset(z_{j_{k}})_{\!k\in\mathbb{N}}\right\} .\label{eq:def_lift_deltax}
\end{align}
Roughly, the set $\hat{\Delta}(\hat{x})$ is formed by all points obtained as accumulation points of a sequence $(\hat y_{j_{k}})_{k\in\mathbb{N}}$ where $\hat y_{j_{k}}\in\hat{f}^{j_{k}}(\pi^{-1}(D(z_{j_{k}},j_{k},\delta)))$ for every $k\ge 1$.
The support of the lift $\hat\nu$ of $\nu$ to $M^f$ is contained in the union of this family of sets. 
More precisely:

\begin{lemma}
Given $\hat{y}\in supp(\hat{\nu})$, there exists $\hat{x}\in M^{f}$
and $\hat{\Delta}(\hat{x})\subset M^{f}$ (given by  \eqref{eq:def_lift_deltax}) 
such that $\hat{y}\in\hat{\Delta}(\hat{x})$ and $\pi|_{\hat{\Delta}(\hat{x})}:\hat{\Delta}(\hat{x})\rightarrow\Delta(\hat{x})$
is a continuous bijection. \label{prop:bijectiveprojectionondisks}
\end{lemma}

Let $\hat{H}_{\infty}$ be the set of points $\hat{x}\in M^{f}$ given by Lemma~\ref{prop:bijectiveprojectionondisks}. Then
 $\mbox{supp}(\hat{\nu})\subset\bigcup_{\hat{x}\in\hat{H}_{\infty}}\hat{\Delta}(\hat{x})$. This is a key step in the proof
that $\nu$ is a $f$-invariant hyperbolic measure . Moreover, 

\begin{theorem}\cite[Theorem~5.1]{CV17}
There exists $C_{3}>1$ so that the following holds: given $(\hat{x},l)$ 
there exists a family of conditional measures $(\hat{\nu}_{\hat{\Delta}})_{\hat{\Delta}\in\hat{\mathcal{K}}_{\infty}(\hat{x},l)}$ of $\hat{\nu}|_{\hat{K}_{\infty}(\hat{x},l)}$ such that $\pi_{*}\hat{\nu}_{\hat{\Delta}}\ll Leb_{\Delta}$, where $\Delta=\pi(\hat{\Delta})$, 
and 
$
\frac{1}{C_{3}}Leb_{\Delta}(B)\leq\pi_{*}\hat{\nu}_{\hat{\Delta}}(B)\leq C_{3}Leb_{\Delta}(B)
$
for every measurable subset $B\subset\Delta$ and for almost every $\hat{\Delta}\in\hat{\mathcal{K}}_{\infty}(\hat{x},l)$.\label{thm:abscontdisint}
\end{theorem}

\subsection{The second volume lemma}

In this subsection we restate and prove a volume lemma for the SRB measure $\mu$ on the attractor $\Lambda$. This
will be crucial to prove, in Subsection~\ref{wGibbs}, that the SRB measure satisfies a non-uniform version of the Gibbs property.
Recall that a SRB measure for $f$ is an invariant measure satisfying that $\pi_{*} \hat{\mu}_{\mathcal{P}(\hat{x})}\ll \Leb_{\pi(\mathcal{P}(\hat{x}))}$, for $\hat{\mu}$ almost every $\hat{x}\in \Lambda^f$, where $\hat{\mu}$ is the lift of $\mu$ to $\Lambda^f$, for any partition $\mathcal{P}$ subordinated to unstable manifolds. Since $\pi(\mathcal{P}(\hat{x}))\subset W^u_{loc}(\hat{x})$, $Leb_{\pi(\mathcal{P}(\hat{x}))}$ denotes the Lebesgue measure inherited by the Riemannian metric on the unstable manifold associated to $\hat{x}$.

Let us fix a partition $\mathcal{P}$ subordinated to unstable manifolds. Since the natural projection $\pi$ is a homeomorphism with its image when restricted to an atom of $\mathcal{P}$ we can define a measure $\hat{m}_{\hat{\gamma}}=(\pi|_{\hat{\gamma}})^{*}Leb_{\hat{\gamma}}$ for all $\hat{\gamma} \in \mathcal{P} $. So, if $\mu$ is a SRB measure then $\hat{\mu}_{\hat{\gamma}} \ll \hat{m}_{\hat{\gamma}}$. By the Radon-Nykodim theorem, there is a measurable function $\rho_{\hat{\gamma}}: \hat{\gamma}\rightarrow \mathbb{R}$, almost everywhere positive and satisfying that:
\begin{equation}\label{eq:RN}
\hat{\mu}_{\hat{\gamma}}(\hat{B})=\int_{\hat{B}}\rho (\hat{x})d\hat{m}_{\hat{\gamma}}(\hat{x}).
\end{equation}
The following lemma shows that the function $\rho_{\hat{\gamma}}$ is uniformly 
bounded away from zero and infinity in local unstable manifolds. More precisely:

\begin{lemma}\label{le:regularity_density}
Let $\mathcal{P}$ be a partition of $\Lambda^f$ subordinated to unstable manifolds. There exists $C>0$ such that
{
$$
e^{-C dist_{\gamma}(x_0,y_0)^\alpha} 
	\leq \frac{\rho_{\hat{\gamma}}(\hat{x})}{\rho_{\hat{\gamma}}(\hat{y})} 
	\leq  e^{C dist_{\gamma}(x_0,y_0)^\alpha},
$$}
for all $\hat{x},\hat{y}\in\hat{\gamma}$, for all $\hat{\gamma}\in\mathcal{P}$, where $\gamma=\pi(\hat{\gamma})$, $x_0 = \pi (\hat{x})$ e $y_0 =\pi (\hat{y})$.
\end{lemma}
\begin{proof}
Let $\hat{\gamma}$ be an atom of $\mathcal{P}$, let $\gamma=\pi(\hat{\gamma})$ and let $n\in\mathbb{N}$ be fixed. Since the lifts of unstable manifolds to $\Lambda^f$ are invariant by $\hat{f}^{-n}$ and the map $\pi|_{\hat{f}^{-n}(\hat{\gamma})}$ is injective, we conclude that $\hat{f}^{-n}(\mathcal{P})$ is also a partition subordinated to unstable manifolds. Denote $\hat{\gamma}_{-n}:=\hat{f}^{-n}(\hat{\gamma})$ and $\rho_n$ the density of $\hat{\mu}_{\hat{\gamma}_{-n}}$ with respect to $\hat{m}_{\gamma}$. 
By ~\eqref{eq:RN} and change of variables,
\begin{align*}
\hat{\mu}_{\hat{\gamma}_{-n}}(\hat{B}) & =\int_{\hat{f^n}(\hat{B})}\rho_n\circ \hat{f}^{-n}(\hat{x})|\det (D\hat{f}^{-n}|_{T_{\pi(\hat{x})}{\gamma}})|d\hat{m}_{\gamma}(\hat{x}).
\end{align*}
Here $D\hat{f}^{-n}(\hat{x})=(Df^n(x_{-n}))^{-1}$. As the disintegrations are unique (in a $\hat \mu$-full measure subset) then $\hat{f}^n_{*}\hat{\mu}_{\hat{\gamma}_{-n}}=\hat{\mu}_{\hat{\gamma}}$, and so
$
\rho(\hat{x})=\rho_n(\hat{f}^{-n}(\hat{x}))\cdot | \det ( D\hat{f}^{-n}(\hat{x})|_{T_{\pi(\hat{x})}\gamma} ) |.
$
Therefore, given $\hat{x},\ \hat{y}\in \hat\gamma$ we get
\begin{equation}
\frac{ \rho( \hat{x} ) }{ \rho( \hat{y} ) }=\frac{ \rho_n (\hat {f}^{-n} ( \hat{x} ) ) }{ \rho_n (\hat {f}^{-n} ( \hat{y} ) ) } \cdot \frac{ \left| \det\left( D\hat{f}^{-n}(\hat{x})|_{T_{\pi(\hat{x})}\gamma} \right) \right|}{ \left| \det\left( D\hat{f}^{-n}(\hat{y})|_{T_{\pi(\hat{y})}\gamma} \right) \right|}. \label{eq:quocient_density}
\end{equation}
Since $\gamma$ is contained in a local unstable manifold and the Jacobian of disks tangent to the cone field is H\"older continuous (recall Proposition \ref{prop:iteratesoftgbundleholder}) we 
conclude that the second quocient in the right side above is bounded above by 
$
e^{ C dist_{\gamma}(x_0,y_0)^\alpha},
$
where $C=L_1 \cdot \frac{e^{c\alpha/2}}{e^{c\alpha/2}-1}$.

Consider the map $\rho:\Lambda^f \rightarrow \mathbb{R}$ given by $\rho(\hat x)=\rho_{{\hat\gamma}_{\hat{x}}}(\hat{x})$, 
where $\hat{\gamma}_{\hat{x}}$ is the atom of $\mathcal{P}$ that contains $\hat{x}$. By Lusin's theorem for every $\vep>0$
there is a compact subset $\hat{\Omega}\subset \Lambda^f$ such that $\rho\mid_{\hat{\Omega}}$ is a (uniformly) continuous function and $\hat{\mu}(\Lambda^f \backslash \hat{\Omega} )<\epsilon$. 
In particular, given $k\in\mathbb{N}$ there exist $\delta_k>0$ such that 
$$
\left| \frac{\rho(\hat{z})}{\rho(\hat{w})}-1\right|<\frac1k,
$$
for all $\hat{z},\hat{w}\in \hat{\Omega}$ with $\hat{d}(\hat{z},\hat{w})<\delta_k$.
We will use the previous estimate for suitable returns of $\hat x$ and $\hat y$.
By Birkhoff ergodic theorem the frequency of visits of $\hat x$ and $\hat y$ to $\hat{\Omega}$
tends to $\hat{\mu}( \hat{\Omega} )\ge 1-\epsilon$. In particular there are infinitely many simultaneous returns 
of $\hat x,\hat y \in \hat \gamma$ to $\hat\Omega$.
Hence
there exists $n_k\in \mathbb{N}$ satisfying that $\hat{f}^{-n_k}(\hat{x}),\ \hat{f}^{-n_k}(\hat{y})\in \hat{\Omega}$ and $\hat{d}(\hat{x},\hat{y})<\delta_k$. 
Therefore,
$$
\left| \frac{\rho(\hat{f}^{-n_k}(\hat{x}))}{\rho(\hat{f}^{-n_k}(\hat{y}))} - 1 \right|<\frac1k
\quad \text{or, equivalently,}
\quad
\left| \frac{\rho_{n_k}(\hat{f}^{-n_k}(\hat{x}))}{\rho_{n_k}(\hat{f}^{-n_k}(\hat{y}))} - 1 \right|<\frac1k.
$$
In consequence we find a sequence $n_k\to \infty$ satisfying 
$$
\lim_{k\to\infty}\frac{\rho_{n_k}(\hat{f}^{-n_k}(\hat{x}))}{\rho_{n_k}(\hat{f}^{-n_k}(\hat{y}))}=1.
$$
Then, the expression \eqref{eq:quocient_density} also can be bounded by $e^{ C dist_{\gamma}(x_0,y_0)^\alpha}$.
The other inequality in the lemma follows by inversion of the roles of $\hat{x}$ and $\hat{y}$ above.
\end{proof}

\begin{proposition} \label{le:volume2}
Assume that the density $\rho:\Lambda^f\rightarrow \mathbb{R}$ satisfies $\frac1{K^*} \le \rho_{\hat{\gamma}}(\hat{y})\leq K^*$ for all 
$\hat{y}\in \hat{B}$.
Let $\delta>0$ be given by the definition of $c$-cone-hyperbolic times.
Given $0<\vep<\delta$ there exists $K_2=K_2(\vep)>0$ so that if $n$ is a $c$-cone-hyperbolic time for $x$ and $F$ is any subspace of dimension $d$ contained in $C(x)$ then
\begin{equation}
K_2(\vep)^{-1} e^{-\log |\det Df^n(x)\mid_F|} \le \mu(\Lambda \cap B(x,n,\vep)) \le K_2(\vep) \label{key}e^{-\log |\det Df^n(x)\mid_F|}.
\end{equation}
\end{proposition}

\begin{proof}
Let $n$ be a $c$-cone-hyperbolic time for $x$ and $\vep>0$ be as above.
If $\hat{\mu}$ is the lift of the SRB measure $\mu$ to $\Lambda^f$ then $\mu=\pi_{*}\hat{\mu}$. Fix a partition $\mathcal{P}$ of $\Lambda^f$ subordinated to the unstable manifolds. Then
\begin{align*}
\mu(\Lambda \cap B(x,n,\vep)) & =\pi_{*}\hat{\mu}(B(x,n,\vep))=\hat{\mu}(\pi^{-1}(B(x,n,\vep)))=\int_{\mathcal{P}}\hat{\mu}_{\hat{\gamma}}(\pi^{-1}(B(x,n,\vep)))d\hat{\mu}(\hat{\gamma}).
\end{align*}
Denote $\hat{B}:=\pi^{-1}(B(x,n,\vep))$. 
Since $ \pi(\hat{\gamma}) $ is contained in some local unstable manifold  we have that  $ \pi(\hat{\gamma}) $ is tangent to the cone field
$C$. In particular, $dist_{\gamma}(x_0,y_0)\leq 2dist_M(x_0,y_0)<4\varepsilon$ for all $\hat{x},\hat{y}\in \hat{\gamma}\cap \hat{B}$.
Using Lemma \ref{le:regularity_density} and that the density of $\mu$ is bounded away from zero and infinity, 
for every $\hat{\gamma}\in\mathcal{P}$ 
there exists $\hat{y}_{\hat{\gamma}}\in \hat{\gamma}$ so that
\begin{align*}
\hat{\mu}_{\hat{\gamma}}(\hat{B}) & = \int_{\hat{B}}\rho_{\hat{\gamma}}(\hat{x})d\hat{m}_{\hat{\gamma}}(\hat{x})
	\leq e^{C(4\varepsilon)^\alpha} \rho_{\hat{\gamma}}(\hat{y}_{\hat{\gamma}})\hat{m}_{\hat{\gamma}}(\hat{B}) 
	\leq K^* e^{C(4\varepsilon)^\alpha}  \hat{m}_{\hat{\gamma}}(\hat{B})
\end{align*}
and, consequently,
$
\mu(\Lambda \cap B(x,n,\vep))  \leq {e^{C(4\varepsilon)^\alpha}} K^* 
\int_{\mathcal{P}} \hat{m}_{\hat{\gamma}}(\hat{B})d\hat{\mu}(\hat{\gamma}),
$
for every sufficiently small $\varepsilon$.
Inequality ~\eqref{eq:dynball1d2} now ensures that
\begin{align*}
\hat{m}_{\hat{\gamma}}(\hat{B}) & =Leb_{\pi(\hat{\gamma})}(B(x,n,\vep)) \leq K(\varepsilon) |\det Df^n(x)\mid_F|^{-1}
\end{align*}
where $F$ is the tangent space to some unstable disc that contains $x$, for all $\hat{\gamma}$ that intersects $\hat{B}$. Therefore,
\begin{equation*}
\mu(\Lambda \cap B(x,n,\vep)) \leq K^* {e^{C(4\varepsilon)^\alpha}}
  K(\vep) e^{-\log |\det Df^n(x)\mid_F|}.
\end{equation*}
The inverse inequality is obtained in the same way, using the inverse inequalities from Lemma \ref{le:regularity_density} and Proposition \ref{le:volume1}. So, taking $K_2(\vep)=K^* {e^{C(4\varepsilon)^\alpha}}
K(\vep)$ we complete the proof of the proposition.
\end{proof}

\section{Large deviations: Proof of Theorem ~\ref{thm:main} }\label{sec:LDP}

\subsection{The SRB measure is a Weak Gibbs measure}\label{wGibbs}

In this subsection we prove that under the mild assumptions in Theorem~\ref{thm:main} the SRB measure is a weak Gibbs measure, 
that is, satisfies the conclusion of Proposition~\ref{co:measure.control} below.
Let $H$ be the set of points with infinitely many $c$-cone hyperbolic times.  
It is a standard argument to prove that the volume lemma together with the non-lacunarity of the sequences of cone-hyperbolic times 
ensures the weak Gibbs measure. More precisely:

\begin{proposition}\label{co:measure.control}
Assume that the sequence of cone-hyperbolic times is non-lacunar for $\mu$-almost every $x$.
Then, for $\mu$-almost every $x\in H$, there exists $\delta>0$ so that given $0<\vep<\delta$ there exists a sequence $(K_n(x,\vep))_n$ of positive
integers such that 
$\lim\limits_{\vep\to 0} \limsup\limits_{n \to \infty}
    \frac1n \log K_n(x,\vep)=0$ for $\mu$-almost every $x$ 
and, for {\bf any} subspace $F_1$ of dimension $d$ contained in $C(x)$, it holds 
\begin{equation*}
K_n(x,\vep)^{-1} e^{-\log |\det Df^n(x)\mid_{F_1}|} \le \mu(\Lambda \cap B(x,n,\vep)) \le K_n(x,\vep) e^{-\log |\det Df^n(x)\mid_{F_1}|}
\end{equation*}
for every $n\ge 1$.
\end{proposition}

\begin{proof}
We will prove the upper bound in the proposition (the lower bound is analogous).
Let  $\delta>0$ be given by the second volume lemma. Then, for every $0<\vep<\delta$ there exists $K_2(\vep)>1$ so that if $n$ is a $c$-cone-hyperbolic time for $x$ 
and $F_1$ is {any} 
subspace of dimension $d$ contained in $C(x)$ then
\begin{equation*}
K_2(\vep)^{-1} e^{-\log |\det Df^n(x)\mid_{F_1}|} \le \mu(\Lambda \cap B(x,n,\vep)) \le K_2(\vep) e^{-\log |\det Df^n(x)\mid_{F_1}|}.
\end{equation*}
Given an arbitrary $n\ge 1$ write $n_i(x) \leq n < n_{i+1}(x)$, where
$n_i(x)$ and $n_{i+1}(x)$ are consecutive $c$-cone hyperbolic times for $x$. Then taking $n_i=n_i(x)$ it follows that
\begin{align*}
\mu(B(x,n,\vep))
         \le \mu(B(x,n_i,\vep)) 
         & \le K_2(\vep) e^{-\log |\det Df^{n_i}(x)\mid_{F_1}|}\\
         & = K_2(\vep) e^{-\log |\det Df^{n}(x)\mid_{F_1}|+\log |\det Df^{n-n_i}(f^{n_i}(x))\mid_{F_1}|}\\
        & = K_n(x,\vep) \, e^{-\log |\det Df^{n}(x)\mid_{F_1}|},
\end{align*}
where 
\begin{equation}\label{relationtail}
K_n(x,\vep):=K_2(\vep) (d^{-1} \min_{y\in M} \| Df(y)^{-1}\|)^{{n-n_i}} 
\end{equation}
 (depends
only on the center $x$). 
Now, since
$c$-cone-hyperbolic times concatenate, note that 
\begin{align*}
\frac1n \log K_n(x,\vep) & \le \frac1n\log K_2(\vep) + \frac{n-n_i(x)}n \log (d^{-1} \min_{y\in M} \| Df(y)^{-1}\|) \\
	& \le \frac1n\log K_2(\vep) + \frac{n_{i+1}(x)-n_i(x)}n \log (d^{-1} \min_{y\in M} \| Df(y)^{-1}\|)
\end{align*}
which, by assumption, tends to zero as $n\to\infty$.
This finishes the proof.
\end{proof}

\begin{remark}
In the case that $\Lambda$ is a hyperbolic set for the endomorphism $f$, it admits a Markov partition. 
If this is the case, since the SRB measure $\mu$ satisfies a weak Gibbs property, several consequences can be obtained from 
\cite{VZ}: exponential large deviation bounds and almost sure estimates for the error term in the Shannon-McMillan-Breiman convergence to entropy,  and a topological characterization of large deviations bounds for Gibbs measures, and deduce 
their local entropy is zero (see \cite{VZ} for the precise statements). 
\end{remark}

\subsection{Large deviations for the SRB measure}

Since the SRB measure is a weak Gibbs measure (Proposition~\ref{co:measure.control}) we 
adapt the methods used in \cite{Var12, VZ0, You98} to obtain large deviations estimates for weak Gibbs measures. 
The main difference here lies on the non-additive 
characterization of volume expansion along the invariant cone field
which in turn requires some ingredients from the variational principle for sub-additive sequences of potentials.

We proceed with the proof of Theorem~\ref{thm:main}.
Let $\beta>0$, $0<\vep<\de_0/2$ and $n \ge 1$ be fixed and let $F\subset \mathbb R$ be a closed interval. 
Given $n\ge 1$ set $B_n=\{x \in M: \frac{1}{n}\sum_{j=0}^{n-1}\phi(f^j(x))\in F\}$
and write
\begin{equation}\label{eq:decompB}
B_n
\subset \{ x\in \Lambda \colon K_n(x,2\vep) > e^{\beta n}\} \cup \left(B_n \cap \{ x\in \Lambda \colon K_n(x,2\vep) \le e^{\beta n}\}\right)
\end{equation} 
where $K_n(x,2\vep)$ is given by ~\eqref{relationtail}.
Set 
$C=d^{-1} \min_{y\in M} \| Df(y)^{-1}\|$. 
Observe that that if $n\ge 1$ is large (depending only on $\vep$)
\begin{align*}
\{ x\in \Lambda \colon K_n(x,2\vep) > e^{\beta n}\} 
	& \subset \{ x\in \Lambda \colon  (n-n_i(x)) \log C > \beta n - \log K_2(2\vep) > \frac\beta2 n \}\\ 
	& \subset \Big\{ x\in \Lambda \colon  n_1(f^{n_i(x)}(x))  > \frac{\beta n}{2\log C} \Big\} \\
	& \subset \bigcup_{k=1}^{n} \Big\{ x\in \Lambda \colon  n_1(f^{k}(x))  > \frac{\beta n}{2\log C} \Big\} 
\end{align*}
where $n_i(x) \le n < n_{i+1}(x)$ are consecutive cone-hyperbolic times for $x$.
Since $\mu$ is $f$-invariant then
\begin{equation}\label{estimateeq}
\mu(\{ x\in \Lambda \colon K_n(x,\vep) > e^{\beta n}\} ) \le 
	n \, \mu\Big( \big\{x\in \Lambda\colon   n_1(x)  > \frac{\beta n}{2\log C} \big\}\Big).
\end{equation}
On the other hand,  if $E_n\subset B_n \cap \{ x\in \Lambda \colon K_n(x,2\vep) \le e^{\beta n}\}$ is a maximal $(n,\vep)$-separated set, $B_n\cap \De_{n}$ is contained in the union of the dynamical balls $B(x,n,2\vep)$
centered at points of $E_n$ and, consequently,
\begin{equation}\label{eq.upper}
\mu(B_n) < 
        n \, \mu\Big( \big\{x\in \Lambda\colon   n_1(x)  > \frac{\beta n}{2\log C} \big\}\Big) 
        + e^{\beta n} \sum_{x \in E_n} e^{-J_n(x)}
\end{equation}
for every large $n$, where $J_n(x)=\log |\det Df^{n}(x)\mid_{F_1}|$ and $F_1$ is any subspace of dimension $d$
in $C(x)$.
Now, consider the probability measures $\si_n$ and
$\eta_n$ given by
$$
\si_n
   =\frac{1}{Z_n} \sum_{x \in E_n} e^{-J_n(x)} \de_x
   \quad\text{and}\quad
\eta_n
    = \frac{1}{n} \sum_{j=0}^{n-1} f^j_* \si_n,
$$
where $Z_n=\sum_{x \in E_n} e^{-J_n(x)}$, and let $\eta$ be a
weak$^*$ accumulation point of the sequence $(\eta_n)_{n}$. It is
not hard to check that $\eta$ is an $f$-invariant probability
measure. Assume $\cP$ is a partition of $M$ with diameter smaller
than $\vep$ and $\eta(\partial\cP)=0$. Each element of $\cP^{(n)}$
contains a unique point of $E_n$. As in the proof of the variational principle 
for sub-additive potentials (cf. \cite[pp. 648--550]{CFH})
$$
H_{\si_n}(\cP^{(n)}) - \int J_n (x) \,d\si_n
    = \log \Big( \sum_{x\in E_n} e^{-J_n(x)} \Big)
$$
which guarantees  
that
\begin{equation}\label{eq.upper2}
\limsup_{n \to \infty}
    \frac{1}{n} \log Z_n
    \leq h_\eta(f) - \lim_{n\to\infty} \frac1n \int J_n(x) \, d\eta(x)
    =  h_\eta(f) -\int \Gamma_\eta(x) \, d\eta.
\end{equation}
We used  the definition
of $\Gamma_\eta$ in ~\eqref{def:Gamma} and Corollary~\ref{cor:subspaces} to guarantee that the limit in the right-hand side 
above exists. Since $\frac1n \sum_{j=0}^{n-1} \phi (f^j(x)) \in F$ for every $x\in E_n$ then $\int \phi \, d\eta_n \in F$ for every $n\ge 1$
and
$\int \phi \, d\eta \in F$ by weak$^*$ convergence.
Finally, equations ~\eqref{relationtail}, \eqref{eq.upper} and \eqref{eq.upper2} imply that
for every $\beta>0$
\begin{align}
\limsup_{n \to \infty} \frac{1}{n} \log \mu(B_n)
   & \leq \max\{{\mathcal E}_\mu(\beta) \; ,  -\inf_{c\in F 
   } I(c)+\beta\} \label{thmA-1}
\end{align}
where
$$
{\mathcal E}_\mu(\beta)=\limsup_{n\to\infty} \frac1n \log \mu\Big(x \colon n_1( x) >\frac{\beta n}{2\log C} \Big)
$$
and 
$$
I(c)= \sup \{ h_\eta(f) - \int \Gamma_\eta(x) \, d\eta \colon \eta\in \mathcal M_1(f),\, \int \phi \, d\eta\in F  \}.
$$
Since ~\eqref{thmA-1} holds and $\beta>0$ is arbitrary, taking the infimum over all possible $\beta>0$
in the second alternative in ~\eqref{thmA-1} we deduce that for every $\beta>0$
\begin{align*}
\limsup_{n \to \infty} \frac{1}{n} \log \mu(B_n)
   & \leq 
   \inf_{\beta>0} \Big[\max\{{\mathcal E}_\mu(\beta) \; ,  -\inf_{c\in F 
   } I(c)+\beta\}\Big]
\end{align*}
which prove the second item in Theorem~\ref{thm:main}. 
This completes the proof of the theorem.

\begin{remark}
Note that 
$$
\frac{{\mathcal E}_\mu(\beta)}{\beta}=\limsup_{n\to\infty} \frac1{\beta n} \log \mu\big(x \colon n_1( x) >\frac{\beta n}{2\log C} \big)
	\le  \frac1{2\log C} \limsup_{n\to\infty} \frac1{n} \log \mu\big(x \colon n_1( x) >n\big),
$$
where the right-hand side is a constant $\mathcal E\le 0$ which independs on $\beta$. In consequence, 
${\mathcal E}_\mu(\beta) \le \beta \mathcal E \le 0$ for every $\beta>0$. 
Hence, the inequality is effective for large deviations when there exists $\beta>0$ so that ${\mathcal E}_\mu(\beta)<0$
or, in other words, when the first cone-hyperbolic time $n_1$ has exponential tail.
\end{remark}

\subsection{Proof of Corollary~\ref{cor:A} }

Assume that we are in the context of Theorem~\ref{thm:main} and that $n_1$ has exponential tails.
On the one hand, there exists $\cE>0$ so that ${\mathcal E}_\mu(\beta) \le -\beta \cE$ for all $\beta>0$. 
On the other hand, it follows from the proof of 
Theorem~\ref{thm:main} that for every $\beta>0$ there exists a probability measure $\eta$
such that  $\int \phi\, d\eta \notin (\int \phi\,d\mu-\delta,\int \phi\,d\mu+\delta)$
and 
\begin{align}\label{eqcoreta}
\limsup_{n \to \infty} \frac{1}{n} \log \mu(B_n)
   & \leq \max\{{\mathcal E}_\mu(\beta) \; ,  h_\eta(f) -\int \Gamma_\eta(x) \, d\eta +\beta\} 
\end{align}
for every $\beta>0$ (recall equations \eqref{eq.upper} - \eqref{thmA-1}).
By Ruelle inequality it follow that 
$$
h_\eta(f) -\int \Gamma_\eta(x) \, d\eta = h_{\eta}(f)-\int \sum_{\lambda_i(\eta,x)>0} \lambda_i(\eta,x) \, d\eta \le 0.
$$
By Pesin's entropy formula (\cite[Theorem VII.1.1]{QXZ09}), the unique SRB measure $\mu$ for $f$ on the attractor $\Lambda$ is the
unique invariant measure that satisfies the equality
$
h_{\mu}(f)=\int \sum_{\lambda_i(\mu,x)>0} \lambda_i(\mu,x) \, d\mu.
$
In consequence, 
there exists $\beta>0$ so that the right-hand side of ~\eqref{eqcoreta} is strictly negative. 
This proves the corollary.

\section{Some examples\label{sec:examples}}

In this section we give some applications  of our main results in the context of Anosov and 
partially hyperbolic endomorphisms. 

\subsection{Anosov endomorphisms}Let $f:M\rightarrow M$ be an Anosov endomorphism, i.e., $f$ be a local diffeomorphism
satisfying (H1)-(H4) on $M$ with condition \eqref{eq:nuexpansion} replaced by $\|(Df(x)|_{C(x)}^{-1}\|\leq e^{-2c}<1$ for all $x\in M$.
Then for all $\hat{x}\in M^f$ there exists a splitting $T_{\hat{x}}M^f=E^s_{x_0}\oplus E^u_{\hat{x}}$ and constants $C>0$ and $\lambda \in (0,1)$ satisfying that: (1) $Df(x_0)\cdot E^s_{x_0}=E^s_{f(x_0)}$ and $Df(x_0)\cdot E^u_{\hat{x}}=E^u_{\hat{f}(\hat{x})}$ (2) $ \| Df^n(x_0)\cdot v \|\leq C\cdot \lambda^n \|v\| $ for all $v\in E^s_{x_0}$ and  $n\in\mathbb{N}$; (3) $\|(Df^n(x_{-n}))^{-1}\cdot v\| \leq C\lambda^n\|v\|$ for all $v\in E^u_{\hat{x}}$ \cite{Prz76}. 
It is clear that this class of endomorphisms fits in the class of dynamics considered here, hence there exists finitely many SRB measures for $f$ and it is unique whenever $f$ is transitive  \cite[Theorem A]{CV17}. 
Observe that, the uniform expansion along the cone field $C$ implies that all times $c$-cone-hyperbolic time. Consequently, the non-lacunarity assumption or the sequence of $c$-cone-hyperbolic times is trivially satisfied.  
Therefore, Theorem~\ref{thm:main} and Corollary~\ref{cor:A} provide large deviations upper bounds for  
for the SRB measure of the Anosov endomorphism.
This provides an alternative proof to the large deviations estimates for uniformly hyperbolic endomorphisms obtained in 
\cite[Theorem 1.2]{LQZ03} with respect to continuous potentials for the SRB measure, which explores
the H\"older continuity of the unstable Jacobian on the natural extension $M^f$.

\subsection{Endomorphisms derived from Anosov}

Dynamical systems in the isotopy class of uniformly hyperbolic ones have been 
intensively studied in the last decades, as candidates for robustly transitive dynamics. 
The class of $C^1$-robustly transitive 
non-Anosov diffeomorphisms considered by  Ma\~n\'e 
are among the first classes of examples of this kind, and their SRB measures were 
constructed by Carvalho in \cite{Car93}. 
In this subsection we illustrate how some partial hyperbolic endomorphisms 
can be obtained, by local perturbations,  in the isotopy class of hyperbolic endomorphisms.
In ~\cite{Sumi} Sumi constructed non-hyperbolic topologically mixing partially hyperbolic endomorphisms on $\mathbb T^2$ .
We make the construction of the example in dimension $3$ for simplicity although similar
statements hold in higher dimension. 
In what follows we give an example of dynamical system with hyperbolic periodic points with different index but we
could consider also the case of existence of periodic points with an indifferent direction. 

Consider $M=\mathbb{T}^{3}$, let $g:M\rightarrow M$ be
a linear Anosov endomorphism induced by a hyperbolic matrix
$A\in \mathcal M_{3\times 3}(\mathbb Z)$ displaying three real eigenvalues and 
such that $TM=E^{s}\oplus E^{u}$ is the $Dg$-invariant splitting, 
where $\dim E^{u}=2$. For instance, take a matrix of the form
$$A=
\left(
\begin{array}{ccc}
n & 1 & 0 \\
1 & 1 & 0 \\
0 & 0 & 2 
\end{array}
\right)
$$
for an integer $n \ge 2$. The map $g$ is a special Anosov endomorphism, meaning that the unstable space 
independs of the pre-orbits and it admits a finest dominated splitting $E^s \oplus E^u \oplus E^{uu}$.
If $p\in M$ is a fixed point for $g$ 
and $\delta>0$ is sufficiently small, 
one can write $g$ on the ball $B(p,\delta)$ (in terms of local coordinates in $E_{p}^{s}\oplus E_{p}^{u}$) by
$
g(x,y,z)=(f(x), h(y,z)), 
$
where $x\in E_{p}^{s}$, $(y,z)\in E_{p}^{u}$, $f$ is a contraction along $E_{p}^{s}$ and
$h$ is expanding along $E^u_p$.
Let $\lambda_{2},\lambda_{3}\in\mathbb{R}$ be the eigenvalues of
$Dg(p)\mid_{E_{p}^{u}}$. Suppose that $\left|\lambda_{2}\right|\geq\left|\lambda_{3}\right|>1$.
Consider an isotopy $\left[0,1\right]\ni t\mapsto h_{t}$ satisfying:
(i) $h_{0}=h$;
(ii) for every $t\in\left[0,1\right]$ the diffeomorphism $h_{t}$ has a fixed point $p_{t}$ (continuation
of $p$) which, without loss of generality, we assume to coincide with $p$;
(iii) $h_{t}: E_{t,p}^{c}\rightarrow E_{t,p}^{c}$
is a $C^{1+\alpha}$ map, where $E_{t,p}^{c}:=E_{p}^{u}$
for every $t\in\left[0,1\right]$;
(iv) the eigenvalues of $Dh_{1}$ at $p$ are $\lambda_{2}$ and $\rho\in\mathbb{R}$
with $\left|\rho\right|<1$ (determined by ~\eqref{rhoeq1} and ~\eqref{eqdefrho} below)
(v) if $g_t(x,y,z)=(f(x), h_t(y,z))$ then $g_{t}\mid_{M\backslash B(p,\delta)}=g\mid_{M\backslash B(p,\delta)}$
for every $t\in [0,1]$.
Reducing $\delta>0$ if necessary, we may assume that $g_{t}\mid_{B(p,\delta)}$
is injective for every $t\in\left[0,1\right]$.

Roughly, the fixed point $p$ goes through a pitchfork bifurcation along a one dimensional 
subspace contained in the unstable subspace associated to the original dynamics 
on the open set $\mathcal{O}=B(p,\delta)$. 
Given $a\in(0,1)$ and the splitting $TM=E_{t}^{s}\oplus E_{t}^{c}$ consider the families of cone fields
\[
\mathcal{C}_{a}^{s}(x):=\left\{ v=v_{s}\oplus v_{c}:\ \left\Vert v_{c}\right\Vert \leq a \left\Vert v_{s}\right\Vert \right\} 
	\quad\text{and}\quad
\mathcal{C}_{a}^{u}(x):=\left\{ v=v_{s}\oplus v_{c}:\ \left\Vert v_{s}\right\Vert \leq  a  \left\Vert v_{c}\right\Vert \right\}
\]
and assume without loss of generality that $\|(u,v,w)\|=|u|+|v|+|w|$ for $(u,v,w) \in E^s\oplus E^u \oplus E^{uu}$ 
(local coordinates for $g$). If $a>0$ is small then the families of 
cone fields are $Dg_1$-invariant. Indeed,
one can write $v=v_{s}\oplus v_{c} \in\mathcal{C}^{s}(g_{1}(p))$, with $v_{*}\in E_{1,g_1(p)}^{*}$,
$*\in\{s,c\}$. 
Assume $\rho$ is such that 
\begin{equation}\label{rhoeq1}
\left|\lambda_{1}\cdot\rho^{-1}\right| < 1.
\end{equation}
Then
$
\left\Vert Dg_{1}(g_1(p))^{-1}\cdot v_{c}\right\Vert  
	  \leq 
	 	\left|\rho\right|^{-1}\left\Vert v_{c}\right\Vert 
	  \leq\left|\rho\right|^{-1} a \left\Vert v_{s}\right\Vert  
	 	\leq\left|\lambda_{1}\right|  \left|\rho\right|^{-1}\cdot a \left\Vert Dg_{1}(p)^{-1}\cdot v_{s}\right\Vert,
$
which proves that $Dg_{1}(p)^{-1}\cdot \mathcal{C}_{a}^{s}(g_1(p)) \subset \mathcal{C}_{|\lambda_1 \rho^{-1}| a}^{s}(p)$.
Analogously, if $v\in\mathcal{C}_{1}^{u}(p)$ then 
$\left\Vert Dg_{1}(p)\cdot v_{s}\right\Vert \leq\left|\lambda_{2}^{-1} \lambda_{1}\right| a \left\Vert Dg_{1}(p)\cdot v_{c}\right\Vert$ and, consequently, 
$Dg_{1}(p)\cdot \mathcal{C}_{a}^{u}(p) \subset \mathcal{C}_{|\lambda_1\lambda_2^{-1}|a}^{u}(g_1(p))$.
By continuity, since $\delta$ is assumed small, we get the invariance of the cone fields for all points on the ball $B(p,\delta)$.
Now observe that $g_{1}\mid_{B(p,\delta)}$ expands volume along $E_{1,x}^{c}$ as it coincides with 
$g\mid_{B(p,\delta)}$. Moreover, if $\rho$ is so 
\begin{equation}\label{eqdefrho}
|\lambda_2\cdot \rho |>1
\end{equation}
then $|\det(Dg_1(p)\mid_{E_{1,p}^{c}})|=|\lambda_2|\, |\rho|>1$
and, using that $a>0$ is small and $x\mapsto \det(Dg_1(x)\mid_{E_{1,x}^{c}})$ is continuous, 
we conclude that  $\inf_{x\in M} |\det(Dg_1(x)\mid_{E_{1,x}^{c}})| >1$ and that the same property holds
for the Jacobian along disks tangent to the cone field.
Simple computations show that $\left\Vert Dg_{1}(p)\cdot v\right\Vert \geq(1- a )\cdot\rho \left\Vert v\right\Vert $
for $v\in\mathcal{C}_{a}^{u}(p)$
and, 
by continuity of the derivative, we get that
$\left\Vert Dg_{1}(x)\cdot v\right\Vert \geq L \left\Vert v\right\Vert $
for every $x\in B(p,\delta)$ and $v\in C^{u}(x)$ with $L=L(\rho)$ close to $(1-a) \rho$, hence close to 1.
If $x\notin B(p,\delta)$, then $g_{1}\mid_{M\backslash B(p,\delta)}=g\mid_{M\backslash B(p,\delta)}$
and so $\left\Vert Dg_{1}(x)\cdot v\right\Vert \geq(1- a )\cdot\lambda_{3}\left\Vert v\right\Vert$ for every 
 $v\in\mathcal{C}_{a}^{u}(x)$.
Finally, since $\delta$ is assumed small, any partition $\left\{ V_{1},V_{2},\dots,V_{k}, V_{k+1}\right\} $ of $\mathbb T^3$ 
such that $V_{k+1}=B(p,\delta)$ and each $V_i$ contains a ball of radius $\delta$ and is contained in a ball of radius $2\delta$ satisfies that for every disk $D$ tangent to the cone field, the image $g_1(D\cap V_i )\cap V_j $ has at most one connected component for all $i,j\in \{ 1,2,\dots, k+1 \}$. 
The estimates in \cite[Section 7]{CV17} ensure that $g_1$ satisfies the hypothesis (H1)-(H4) and that there exist
$K,\vep>0$ uniform so that for every disk $D$ tangent to the cone field 
$$
Leb_D (n_1(\cdot)> n) \le K e^{-\vep n}
$$
for every $n\ge 1$. 
Then, a calculation identical to  \cite[Lemma~4]{VV10} guarantees that the density of $\mu$ with respect to the Lebesgue measure
along local unstable manifolds is bounded away from infinity, hence the first cone-hyperbolic time is integrable
with respect to $\mu$. Therefore, the sequence of $c$-cone-hyperbolic times is non-lacunar $\mu$-almost everywhere 
\cite[Corollary~3.8]{VV10}. As a consequence of our results the volume lemmas hold and we obtain exponential convergence
of Birkhoff averages for $g_1$ and all continuous observables with respect to the SRB measure.

\subsection*{Acknowledgments:} 
The second author was partially supported by a INCTMAT-CAPES Postdoctoral fellowship at Universidade Federal da Bahia. 
The third author was partially supported by CNPq-Brazil. 


\end{document}